\documentclass[11pt]{article}
\usepackage[margin=1in]{geometry}
\usepackage{amsmath,amssymb,amsthm}
\usepackage{mathtools}
\usepackage{array}
\usepackage{enumitem}
\usepackage{hyperref}

\renewcommand{\arraystretch}{0.9}
\makeatletter
\renewcommand*{\env@matrix}[1][\arraystretch]{%
  \edef\arraystretch{#1}%
  \hskip -\arraycolsep
  \let\@ifnextchar\new@ifnextchar
  \array{*\c@MaxMatrixCols c}}
\makeatother

\newcommand{\sbmat}[1]{\left[\begin{smallmatrix}#1\end{smallmatrix}\right]}

\theoremstyle{plain}
\newtheorem{theorem}{THEOREM}[section]
\newtheorem{lemma}[theorem]{LEMMA}

\title{Linear Algebra and Algebraic Geometry: A Matrix Construction for Classifying Families of Algebraic Curves with No Solutions in $\mathbb{Z}^+$}
\author{Shazali Abdalla Fadul \\ \small Faculty Of Mathematics Sciences \& Statistics, AL-neelain University, Khartoum, Sudan \\ \small e-mail address: shazlyabdullah3@gmail.com}
\date{}

\begin{document}
\maketitle

\begin{abstract}
A substantial body of results is known for curves of degree $d=3$, namely elliptic curves, including the theorems of Siegel, Mazur, and Mordell, among others. More generally, Faltings established comprehensive results for all algebraic curves of degree $d\ge 2$. However, these results do not, in general, furnish an explicit or systematic procedure for classifying algebraic curves that fail to admit solutions over a prescribed set such as $\mathbb{Z}^+$. The principal aim of the present work is accordingly to develop a method for classifying families of algebraic curves having no solutions in $\mathbb{Z}^+$. We begin by establishing a correspondence between algebraic geometry and linear algebra, from which we develop a classification procedure grounded in linear-algebraic constructions and tools.

Specifically, let $S: Ax = b$ be a linear system, where $A \in M_{3k \times m}(\mathbb{Z})$ satisfies suitable conditions on its entries, and let $S = \{s_1, s_2, \dots, s_N\} \subset \mathbb{Z}^m$ denote the solution set of the system, with $A = (a_{ij})$, and $s_i = (c_1, c_2, \dots, c_{n+1})$. From such a solution we construct algebraic curves of degree $n \le m$, nonsingular curves $C_{j s_i}$ of genus $g \ge 0$, of the form
\[
C_{j s_i} : Y^2 = a_{3j1} c_1 X^n + a_{3j2} c_2 X^{n-1} + \cdots + a_{3jn} c_n .
\]
If $y \ge 1$, $x > 1$ and $\forall (x,y) \in C_{j s_i}(\mathbb{Z}^+)$ holds then
\[
\Big[ C_{j s_i}(\mathbb{Z}^+) \Big]^{N}_{\substack{i=1\\ 1\le j \le k}} = \emptyset ,
\]
 if the system $S$ admits infinitely many solutions, then
\[
\Big[ C_{j s_i}(\mathbb{Z}^+) \Big]^{\infty}_{\substack{i=1\\ 1\le j \le k}} = \emptyset \quad \text{if } N \to \infty .
\]

We further extend this correspondence to several systems considered simultaneously. In particular, we consider three systems $S_j : A_j x = b_j$, $j=1,2,3$, representing linear, quadratic, and cubic conditions, respectively, with integer matrices. When these systems admit a common solution set $S_1 \cap S_2 \cap S_3 = \{s_1, s_2, \dots, s_N\} \subset \mathbb{Z}^m$ and $s_i = (x_1, x_2, \dots, x_m)$, we construct an algebraic curve of degree $\deg C_{s_i} = \max(x_1, x_2, \dots, x_m)$ and whose equation has the form
\[
C_{s_i} : DY^2 = a_1 X^{x_1} + a_2 X^{x_2} + \cdots + a_m X^{x_m}.
\]
If $x > 1$, $y \ge 1$ and $\forall (x,y) \in C_{s_i}(\mathbb{Z}^+)$ hold, then $\{C_{s_i}(\mathbb{Z}^+)\}^N_{i=1} = \emptyset$.

Several equivalent formulations of the conditions imposed on these systems are also given. The method is then generalized to common solutions of three algebraic curves defined by multivariable equations $C_1, G_2, F_3 \in \mathbb{Z}(x_1,\dots,x_n)$. From a common solution of these three curves we construct a fourth algebraic curve in two variables $C_P \in \mathbb{Z}(x,y)$, having no solution in $\mathbb{Z}^+$. Finally, we study intersections of algebraic curves. We prove that, given a solution of a linear equation of the form
\[
M y^2 = -b_0 x_0 + b_1 x_1 - b_2 x_2 - \cdots - b_{2m-1} x_{2m-1} - b_{2m} x_{2m}
\]
by $S_M = \{P_1, P_2, \dots, P_N\} \in \mathbb{Z}^{m+2}$, one can construct three algebraic curves, denoted by $C_{P_i}, G_{P_i}, F_{P_i} \in \mathbb{Z}(x,y)$, whose common intersection over $\mathbb{Z}^+$ is empty. Then
\[
\Big\{ C_{P_i}(\mathbb{Z}^+) \cap G_{P_i}(\mathbb{Z}^+) \cap F_{P_i}(\mathbb{Z}^+) \Big\}^{\infty}_{i=1} = \emptyset \quad \text{if } N \Rightarrow \infty .
\]
The forms and degrees of the three curves thus constructed are described explicitly.

Keywords \textbf{:} classifying algebraic curves, linear system, algebraic curves, intersections.
\end{abstract}

\section{Introduction}

Algebraic Geometry and Arithmetic Geometry are among the most important subjects in modern mathematics. They are among the most active branches due to their connections with several fields of mathematics, such as Algebraic Topology, Complex Analysis, and Algebraic Number Theory. In general, Algebraic Geometry studies the zeros or roots of polynomials in several variables, $f(x_1,x_2,x_3,\dots,x_n) \in \mathbb{K}(x_1,x_2,x_3,\dots,x_n)$. The roots lie in algebraically closed fields, such as the field of complex numbers $\mathbb{C}$, or in non-closed fields, such as the field of rational numbers $\mathbb{Q}$ or a finite field $F_p$. Arithmetic Geometry, on the other hand, is a subfield of Algebraic Geometry that uses techniques from Algebraic Geometry to study solutions to Diophantine equations over the sets of rational numbers $\mathbb{Q}$ and integers $\mathbb{Z}$. Mathematicians call the polynomials and quadratic forms ``Diophantine Equations'' if they are solvable over $\mathbb{Z}$, $\mathbb{Q}$, named after the Alexandrian Diophantus who studied polynomial solutions over $\mathbb{Q}$. Pierre de Fermat, the French mathematician, studied Diophantus' work and proved many results, developing proof methods such as the method of infinite descent. Others, such as Euler, Gauss, and Riemann, also studied Diophantine equations, and in the 20th century these equations were studied in an advanced manner using modern techniques developed within Algebraic Geometry. Diophantine equations, algebraic curves, and surfaces were studied by Henri Poincar\'e, Andrew Wiles \cite{ref11}, Louis Mordell \cite{ref12}, C.L. Siegel \cite{ref6}, G. Faltings \cite{ref13}, and others.

Elliptic curves are algebraic curves of degree three in two variables given by $y^2 = f(x)$ where $f(x)$ is a cubic polynomial. Fermat studied the solutions of such curves over $\mathbb{Q}$ and observed that the solutions of these curves have algebraic properties. He proved that if $P_1$ is a point on the curve and there exists another point $P_2$, one can produce a third point $P_1+P_2=R$. This is called the point addition law on curves. This method is called the ``chord-and-tangent method'' for generating solutions. Fermat was the first to use this method in studying solutions of elliptic curves. For more, see \cite{silverman1}. Euler also studied some Diophantine equations. For example, it was proven that the equation $y^2 = x^3+1$ has no solutions except for one, namely $y=\pm 3$ and $x=2$. The finiteness of integral solutions on elliptic curves was studied by Siegel \cite{ref6}. In 1929, Siegel proved a strong result on integral solutions: Let $E$ be an elliptic curve given by the equation $E: y^2 = x^3+ax+b$, $a,b\in\mathbb{Z}$. Then $E$ has only finitely many points $P \in E(\mathbb{Z})$. In other words, the set of integral coordinates $(x,y)$ on $E(\mathbb{Z})$ is finite \cite{ref5b}.

Also, Nagell and Elisabeth Lutz \cite{ref7} studied torsion points on elliptic curves and proved the Nagell--Lutz Theorem: If $E$ is a nonsingular cubic curve $E: y^2 = x^3+Ax^2+Bx+C$ with integer coefficients $A,B,C \in \mathbb{Z}$, and let $D$ be the discriminant of the cubic polynomial, then if $P=(x,y)$ is a rational point of finite order for the group law on the elliptic curve, then $(x,y) \in E(\mathbb{Z})$, or $y=0$ in which case $P$ is a point of order 2, or $y \mid D$, $y^2 \mid D$ \cite{ref5,ref6}. Moreover, Mazur \cite{ref5b} studied torsion points and answered questions posed since 1906--1911. In those years, Beppo Levi published a series of papers investigating the possible finite orders of points on elliptic curves over $\mathbb{Q}$. He showed that there are infinitely many elliptic curves over $\mathbb{Q}$ that have torsion points with a finite group. Mazur proved the Torsion Theorem for elliptic curves:

\begin{theorem}[Mazur, 1977]
The torsion subgroup $E(\mathbb{Q})_{tors}$ of the group of rational points $E(\mathbb{Q})$ on an elliptic curve must be isomorphic to one of the following 15 groups:
\[
C_N \text{ with } 1 \le N \le 10 \text{ or } N=12,
\]
\[
C_2 \times C_{2N} \text{ with } 1 \le N \le 4,
\]
in particular, $|E(\mathbb{Q})_{tors}| \le 16$.
\end{theorem}

Henri Poincar\'e studied rational points using algebraic methods. He investigated the algebraic structure of the rational points on elliptic curves. In 1901 he posed a question about the algebraic structure of the set of solutions on algebraic curves. This question was answered by Louis Mordell in 1922, who proved Mordell's Theorem, a fundamental result in Diophantine Geometry:

\begin{theorem}[Mordell, 1922, \cite{ref12}]
Let $E$ be an elliptic curve given by an equation $E: y^2=x^3+ax+b$, $a,b \in \mathbb{Q}$. Then the group of rational points $E(\mathbb{Q})$ is a finitely generated abelian group. In other words, there exists a finite set of points $P_1,P_2,\dots,P_n \in E(\mathbb{Q})$ such that every point $P \in E(\mathbb{Q})$ can be written in the form
\[
P = n_1 P_1 + n_2 P_2 + n_3 P_3 + \cdots + n_t P_t
\]
for some $n_1,n_2,n_3,\dots,n_t \in \mathbb{Z}$.
\end{theorem}

In 1983, Gerd Faltings \cite{ref13} proved a powerful theorem on algebraic curves that is comprehensive and depends on genus, covering all algebraic curves. Here we state part of Faltings' Theorem, which answers the question of whether the number of rational solutions to algebraic curves is finite or infinite. Let $C$ be a non-singular algebraic curve of genus $g \ge 2$. Then $C(\mathbb{Q})$, the set of rational points on the curve, is finite. According to Faltings' Theorem, a curve of genus $g>1$ has only finitely many rational points. This means that any curve of degree $d \ge 4$ has only finitely many rational solutions. In other words, the set of rational solutions for curves of degree greater than three is finite. The theorem also includes implications for curves of degree two and three.

Siegel's theorem provides a strong finiteness result by establishing that the integral points on an elliptic curve are finite. Likewise, the Nagell--Lutz theorem gives a precise description of the algebraic properties of torsion points and determines the possible structure of the torsion subgroup. Mazur's torsion theorem further clarifies the possible types and structures of torsion subgroups over the rational numbers, providing a complete classification of the torsion subgroup of an elliptic curve over $\mathbb{Q}$. On the other hand, Mordell's theorem constitutes one of the fundamental results in Diophantine geometry: it shows how the rational solutions of elliptic curves are organized. More precisely, for every elliptic curve over $\mathbb{Q}$, there exists a finite collection of rational points that generates all the remaining rational points. Finally, Faltings' theorem provides a comprehensive finiteness principle for rational points on algebraic curves of degree greater than one and reveals important structural properties of algebraic curves. These results have exposed many of the structural features and fundamental properties of algebraic curves. Nevertheless, the aforementioned results presuppose the existence of solutions on the algebraic curve and subsequently describe the properties of those solutions.

This raises a natural question: what happens when an algebraic curve has no solutions in a prescribed set? How can one prove that a given algebraic curve has no solutions in such a set? The results cited above do not provide a direct answer to this problem. In this respect, they fail to offer a general criterion for proving the nonexistence of solutions. Although several results addressing particular cases are known, they remain restricted to special classes of equations or curves \cite{ref10,ref15,ref9,ref17}. Therefore, the aim of this work is to introduce a method for classifying families of algebraic curves that are insoluble in a prescribed set, namely, the set of positive integers, by means of a general procedure applicable to all algebraic curves of degree greater than or equal to two, $d\ge 2$. Our approach is based on establishing a connection between algebraic geometry and linear algebra, and then reducing more complicated problems in algebraic geometry to comparatively simpler problems in linear algebra through a suitable matrix construction.

The aim of this study is to classify families of algebraic curves that have no solutions; the results in \cite{ref18} have been extended to include all algebraic curves, as well as to classify some algebraic curves whose solutions have no intersection in the set of positive integers. The results are comprehensive for all algebraic curves of degree greater than or equal to two, $d\ge 2$. We presented a classification method based on linear algebra and algebraic geometry to achieve this. We presented three methods for classifying algebraic curves that are not solvable, as well as methods for classifying curves that have no intersection between their solutions.

For example, we first proved: let $S: Ax=b$ be a system of linear equations having solutions in $\mathbb{Z}$, where $A \in M_{3k\times m}(\mathbb{Z})$, with some conditions on the entries of the matrix; for example, every three consecutive entries in any column are equal, $a_{3i-2,j}=a_{3i-1,j}=a_{3i,j}$, together with sign differences and other conditions. If $S=\{s_1,s_2,\dots,s_N\} \subset \mathbb{Z}^{m+1}$ are the solutions of the system, where $s_i=(c_1,c_2,\dots,c_{n+1})$, there exist algebraic curves of degree $n\le m$, nonsingular curves $C_{j s_i}$ of genus $g\ge 0$, where
\[
C_{j s_i}: Y^2 = a_{3j1}c_1 X^n + a_{3j2}c_2 X^{n-1} + \cdots + a_{3jn}c_n.
\]
If $y \ge 1$, $x>1$ and $\forall (x,y) \in C_{j s_1}(\mathbb{Z}^+)$ holds then
\[
\Big[C_{j s_i}(\mathbb{Z}^+)\Big]^{N}_{\substack{i=1\\ 1\le j\le k}} = \emptyset.
\]
If the linear system has infinitely many suitable solutions, the same construction yields infinitely many algebraic curves with no positive integral solutions,
\[
\Big[C_{j s_i}(\mathbb{Z}^+)\Big]^{\infty}_{\substack{i=1\\ 1\le j\le k}} = \emptyset \quad \text{if } N \to \infty.
\]

The second approach extends the established relationship between algebraic curves and linear equations to quadratic and cubic forms. Let there be three systems: a linear system $S: Ax=b$, where $A \in M_{2\times m}(\mathbb{Z})$; a quadratic system $Q: Bx=d$ where $B \in M_m(\mathbb{Z})$ consisting of $m$ terms, for example, in the form $Q: 2k = a_1 x_1^2 + a_2 x_2^2 + \cdots + a_m x_m^2$; and a cubic system $V: Dx=q$ where $D \in M_m(\mathbb{Z})$, for example, in the form $V: 4s-1 = a_1 x_1^3 + a_2 x_2^3 + \cdots + a_m x_m^3$. Suppose that $S_S, S_Q, S_V$ are the solution sets of the three systems, respectively. If a common solution $S_S \cap S_Q \cap S_V = \{s_1,s_2,\dots,s_N\} \subset \mathbb{Z}^m$ exists and $s_i=(x_1,x_2,\dots,x_m)$, then one can construct a curve, nonsingular curve $C_{s_i}$, of genus $g\ge 0$, where
\[
C_{s_i}: D Y^2 = a_1 X^{x_1} + a_2 X^{x_2} + \cdots + a_m X^{x_m}.
\]
If $x>1$, $y\ge 1$ and $\forall (x,y) \in C_{s_i}(\mathbb{Z}^+)$ holds then $\{C_{s_i}(\mathbb{Z}^+)\}^{N}_{i=1} = \emptyset$, and $\deg C_{s_i} = \max(x_1,x_2,\dots,x_m)$.

The third approach generalizes these ideas to the relationship between polynomial algebraic curves and algebraic curves in two variables. Let $T=\{C_1,C_2,C_3\}\subset \mathbb{Z}(x_1,\dots,x_n)$ and suppose that three algebraic curves have a common integral solution, so that $V(T) = \{P \in \mathbb{A}^n_{\mathbb{Z}} \mid C_i(P)=0 \text{ for all } C_i \in T\}$. Then a fourth curve $C_P$ can be constructed. If $x>1$ and $y\ge 1$ and $\forall (x,y) \in C_P(\mathbb{Z}^+)$, then the resulting curve satisfies $C_P(\mathbb{Z}^+) = \emptyset$.

The second main topic concerns intersections of solution sets of algebraic curves. We establish a connection between linear equations and the intersections of solution sets of certain algebraic curves. Let $S_M(\mathbb{Z}) = \{s_1,s_2,\dots,s_N\} \subset \mathbb{Z}^{m+2}$ be a solution of the linear equation $S_M: MY^2 = -a_0 x_0 + a_1 x_1 - \cdots - a_{2m} x_{2m}$. If three algebraic curves share certain coefficients occurring in this linear equation, denoted by $C_{P_i}, G_{P_i}, F_{P_i} \in \mathbb{Z}(x,y)$, and $C_{P_i}(\mathbb{Z}^+), G_{P_i}(\mathbb{Z}^+), F_{P_i}(\mathbb{Z}^+)$ are positive integer solutions, and if $x>1$, $y\ge 1$ and $\forall (x,y) \in C_{P_1}(\mathbb{Z}^+), G_{P_2}(\mathbb{Z}^+), F_{P_3}(\mathbb{Z}^+)$, their solution sets have empty common intersection, that is, $C_{P_i}(\mathbb{Z}^+)\cap G_{P_i}(\mathbb{Z}^+)\cap F_{P_i}(\mathbb{Z}^+) = \emptyset$. If $|S_M(\mathbb{Z})| = N$ then
\[
\Big\{C_{P_i}(\mathbb{Z}^+)\cap G_{P_i}(\mathbb{Z}^+)\cap F_{P_i}(\mathbb{Z}^+)\Big\}^{N}_{i=1} = \emptyset.
\]
Moreover, if the linear equation admits infinitely many solutions, then this construction yields infinitely many algebraic curves whose solution sets have no common intersection. Thus, the corresponding family satisfies the required disjointness condition, expressed as
\[
\Big\{C_{P_i}(\mathbb{Z}^+)\cap G_{P_i}(\mathbb{Z}^+)\cap F_{P_i}(\mathbb{Z}^+)\Big\}^{\infty}_{i=1} = \emptyset \quad \text{if } N \Rightarrow \infty.
\]

The work is organized as follows. In Section 2, solutions of the function $M(y^2+\alpha^2) = (x+1)f(x)$ where $f(x)\in\mathbb{Z}(x_1,\dots,x_n)$ are studied using primes of the form $p=4n\pm 1$. This analysis leads to THEOREM 2.3. In Section 3, the relationship between solutions of algebraic curves and the function $M(y^2+\alpha^2)=(x+1)f(x)$ is investigated, resulting in THEOREM 3.1. THEOREM 2.3 and 3.1 are then combined to obtain THEOREM 3.2, which constitutes the principal result of the present work. Section 4 is devoted to the relationship between linear algebra and integral solutions of algebraic curves. We prove that, whenever the system $S$ has a solution, the solutions of $S$ can be used to construct algebraic curves having no positive integral solutions. This yields THEOREM 4.1. We then establish the properties and conditions that the linear system $S$ must satisfy in THEOREMS 4.2, 4.3, and 4.5. In Section 5, the linear relationship between algebraic curves is extended to quadratic and cubic forms. The properties of these equations are described and analyzed using tools from linear algebra. The corresponding results are established in THEOREM 5.1, 5.2, and 5.3. In Section 6, the preceding relationship is generalized to polynomial algebraic curves and the polynomial $f$. The structure and relevant properties of these polynomial curves are studied and clarified in THEOREMS 6.1, 6.2 and 6.3. Finally, Section 7 addresses the second main topic of the work, namely, intersections of solution sets of algebraic curves. We prove that, if a linear equation has a solution, then three algebraic curves can be constructed from the solutions of that equation in such a way that the resulting curves have no common intersection. Furthermore, if the linear equation has infinitely many solutions, then infinitely many algebraic curves can be constructed whose solution sets have no common intersection. The structure of the linear equation and the resulting algebraic curves is described in the corresponding THEOREMS 7.1 and 7.2.

\subsection{Proof Methodology}

The proof methodology is arithmetic; we employ the same techniques as those used in \cite{ref18}, but in a generalized form. We first study properties of the solutions associated with the function $M(y^2+\alpha^2) = (x+1)f(x)$ where $f(x)\in\mathbb{Z}(x_1,\dots,x_n)$, using primes of the form $p=4n\pm 1$, leading to THEOREM 2.3. We then connect $M(y^2+\alpha^2) = (x+1)f(x)$ with solutions of algebraic curves in THEOREM 3.1. Combining these two results yields the strong statement of THEOREM 3.2 concerning the properties of the solutions. THEOREM 4.1 follows from THEOREM 3.2, and the remaining results are derived successively by the same underlying approach.

The argument therefore proceeds through arithmetic properties, their transfer to algebraic curves via $M(y^2+\alpha^2) = (x+1)f(x)$, and successive constructions that extend the framework to linear systems, higher degree forms, polynomial curves, and intersections.

\section{The Method of Proof}

In this section, we establish the fundamental results on which the proof of the main results in this study relies. First, we prove THEOREM 2.1, a result concerning the properties of prime numbers of the form $p=4n-1$. We also state Fermat's theorem on the sum of two squares. We use these results in proving THEOREM 2.3. These findings are crucial for the proof of the principal results.

\begin{theorem}\label{thm2.1}
Let $n = 4m-1$ and $m>0 \in \mathbb{N}$. Then we find $p$ is a prime where $p\mid n$ and $p\equiv -1 \pmod 4$, $p\le n$.
\end{theorem}

\begin{proof}
Let $n=4m-1 \in \mathbb{N}$; so, according to the fundamental theorem of arithmetic (see \cite[Ch.\,1, p.\,83, Thm.\,3,4]{tattersall}), we have $n = p_1^{e}\cdot p_2 \cdots p_n$. We conclude from this that if $p_1^{e}, p_2,\dots,p_n$ are all of the form $p_j=4n+1$, then $n = 4m-1 = 4w+1 = p_1^{e}\cdot p_2\cdots p_n$; this is a contradiction. Therefore, there is at least one prime number of the form $p=4k-1$ where $n=p(4w+1)=4m-1$, so $p\le n$.
\end{proof}

\begin{theorem}[Fermat's theorem on the sum of two squares]\label{thm2.2}
Let $m$ be a natural number that is not a perfect square. Then $m=a^2+b^2$ if and only if all prime factors of $m$ are not of the form $p=4n-1$.
\end{theorem}

\begin{proof}
See \cite[Ch.\,8, p.\,242]{tattersall}. Suppose all prime divisors of $m$ are not of the form $4n-1$. If $m=1$ then $m=1^2+0^2$, and if $m>1$ then we have $m=\prod_{i=1}^{r} p_i$. Now, if $p_i=2$ for each $i=1,2,\dots,r$ then we find that $p_i=1^2+1^2$, and if $p_i\equiv 1\pmod 4$ for all $i$ we get $p_i=a_i^2+b_i^2$ according to Fermat's theorem on the sum of two squares (see \cite[Ch.\,8, p.\,242]{tattersall}), but
\[
p_1 p_2 = (a_1^2+b_1^2)(a_2^2+b_2^2) = (a_1 a_2+b_1 b_2)^2+(a_1 a_2-b_1 b_2)^2.
\]
So, by induction on $r$, we can prove that
\[
\prod_{i=1}^{r} p_i = a^2+b^2. \qedhere
\]
\end{proof}

\begin{lemma}\label{lem2.1}
Let $M,\alpha \ge 1 \in \mathbb{N}$ be integers, where $\alpha, M$ are odd numbers and constants, for all prime divisors of $M,\alpha \ge 1$ of the form $p_j = 4n+1$. Then
\[
M(y^2+\alpha^2) \ne (x+1)f(x) \quad \text{if } x=4m-1 \text{ or } 4m+2 \text{ and } (x,y) \in \mathbb{Z}^+\times\mathbb{Z} \text{ where } \forall f(x)\in\mathbb{Z}(x).
\]
\end{lemma}

\begin{proof}
Let $M,\alpha\ge 1 \in \mathbb{N}$ be integers where $\alpha, M$ are odd numbers, constants for all prime divisors of $M,\alpha$ of the form $p_j=4n+1$, and
\[
M(y^2+\alpha^2) = (x+1)f(x) \quad \text{where } \forall f(x)\in\mathbb{Z}(x).
\]
Let $x=4m-1$ and $y\ge 1 \in \mathbb{Z}$. Then
\[
M(y^2+\alpha^2) = (4m)f(4m-1) \quad \text{where } \forall f(4m-1) \in \mathbb{Z}(x).
\]
Where $M$ is an odd number, thus $M\not\equiv 0\pmod 4$; we deduce from this that $y^2+\alpha^2 \equiv 0\pmod 4$, but $\alpha=2k+1$ is an odd number, thus $y^2+(2k+1)^2 \not\equiv 0\pmod 4$ for all $y\in\mathbb{Z}$. Therefore we obtain $M(y^2+\alpha^2)\ne (4m)f(x)$, and then
\[
M(y^2+\alpha^2) \ne (x+1)f(x) \quad \text{if } (x=4m-1,\ y)>1 \in \mathbb{Z}^+\times\mathbb{Z}.
\]
Let $x=4m+2$ and $y\ge 1\in\mathbb{Z}$. We obtain from that
\[
M(y^2+\alpha^2) = (4m+2+1)f(4m+2) = (4(m+1)-1)f(x) = (4k-1)f(4m+2).
\]
Then
\[
M(y^2+\alpha^2) = (4k-1)f(4m+2).
\]
According to THEOREM~\ref{thm2.1}, this means there is a prime number that divides $4k-1$; therefore $4k-1 = qm'$ where $q=4n-1$. Therefore
\[
M(y^2+\alpha^2) = (qm')f(4m+2).
\]
According to the conditions, $M$ is an odd number, and all its divisors are of the form $p=4n+1$, therefore $M\not\equiv 0\pmod q$, which means $y^2+\alpha^2 \equiv 0\pmod q$. However, according to THEOREM~\ref{thm2.2}, all divisors of $y^2+\alpha^2$ are of the form $p=4n+1$, so we conclude $y^2+\alpha^2\not\equiv 0\pmod q$. Therefore $M(y^2+\alpha^2)\ne (qm')f(4m+2)$, and this means
\[
M(y^2+\alpha^2) \ne (x+1)f(x) \quad \text{if } (x=4m+2, y)>1 \in \mathbb{Z}^+\times\mathbb{Z}.
\]
Then we have that
\[
M(y^2+\alpha^2) \ne (x+1)f(x) \quad \text{if } x=4m-1, 4m+2 \text{ and } (x,y)\in(\mathbb{Z}^2)^{+} \text{ where } \forall f(x)\in\mathbb{Z}(x). \qedhere
\]
\end{proof}

\begin{theorem}\label{thm2.3}
Let $a_0,a_1,\dots,a_n,\alpha,M \in \mathbb{Z}$ be integers, and $M,\alpha\ge 1$ be odd positive integer constants, for all prime divisors of $M,\alpha = p_1\cdot p_2\cdot p_3\cdots p_n$ of the form $p_j=4n+1$. Then we have
\[
M(y^2+\alpha^2) \ne (x+1)f(x) \quad \text{for all } (x,y)\in\mathbb{Z}^+\times\mathbb{Z}, \text{ where } x>1,\ y\ge 1,
\]
if
\[
f(x) = \sum_{j=0}^{n-1} a_j \sum_{i=0}^{n-j-1} \binom{n-j}{i}(x+1)^{n-j-i-1}(-1)^i,
\]
\[
\sum_{j=0}^{n-1} a_j(-1)^{n-j+1} = 4k-1,
\]
\[
\sum_{j=0}^{n-3} (-1)^{n-j-2} a_j (n-j)(n-j-2) + a_{n-1} = 4U-1,
\]
\[
U,K\ge 0\in\mathbb{Z}, \qquad \sum_{j=0}^{n} (-1)^{n-j+1} a_j = M\alpha^2 .
\]
\end{theorem}

\begin{proof}
Let $a_0,a_1,\dots,a_n,M,\alpha\in\mathbb{Z}$ be positive integers, where $\alpha,M\ge 1$ is an odd positive integer constant for all prime divisors of $\alpha,M = p_1\cdot p_2\cdot p_3\cdots p_n$ of the form $p_j=4n+1$. Let us have the equation
\begin{equation}
M(y^2+\alpha^2) = (x+1)f(x) \tag{2.1}
\end{equation}
where
\[
f(x) = \sum_{j=0}^{n-1}\sum_{i=0}^{n-j-1} a_j \binom{n-j}{i}(x+1)^{n-j-i-1}(-1)^i, \qquad \sum_{j=0}^{n} (-1)^{n-j+1}a_j = M\alpha^2.
\]
According to LEMMA~\ref{lem2.1}, if $x=4m-1, 4m+2$ and $(x,y)\in\mathbb{Z}^+\times\mathbb{Z}$ where $\forall f(x)\in\mathbb{Z}(x)$, we obtain
\begin{equation}
M(y^2+\alpha^2) \ne (x+1)f(x). \tag{2.2}
\end{equation}
Let the variables in equation (2.1) take the values $M=M$ and $(x=4m-1 \text{ or } 4m+2, y)\in(\mathbb{Z}^2)^+$; then $M(y^2+\alpha^2)=M(y^2+\alpha^2)$, and notice that equation (2.2) holds for all $\forall f(x)\in\mathbb{Z}(x)$, so where $f(x)=f(x)$. Now, from equations (2.1) and (2.2), we obtain that
\begin{equation}
M(y^2+\alpha^2) \ne (x+1)f(x) \quad \text{if } x=4m-1 \text{ or } 4m+2, \text{ where } (x,y)\in\mathbb{Z}^+\times\mathbb{Z}. \tag{2.3}
\end{equation}
We proved this first case when $(x,y)\in\mathbb{Z}^+\times\mathbb{Z}$ if $x=4m-1,4m+2$. Now we prove the second case when $(x=4m+1\vee 4m, y)\in\mathbb{Z}^+\times\mathbb{Z}$. Let $(x=4m+1,y)\in\mathbb{Z}^+\times\mathbb{Z}$ in equation (2.1); from that we get
\begin{equation}
M(y^2+\alpha^2) = (4m+2)f(4m+1) \tag{2.4}
\end{equation}
where
\[
f(4m+1) = \sum_{j=0}^{n-1} a_j \sum_{i=0}^{n-j-1}\binom{n-j}{i}(4m+2)^{n-j-i-1}(-1)^i.
\]
Then we have that
\begin{align}
f(4m+1) &= \sum_{j=0}^{n-1} a_j \sum_{i=0}^{n-j-1}\binom{n-j}{i}(4m+2)^{n-j-i-1}(-1)^i \notag\\
&= \sum_{j=0}^{n-3} a_j \sum_{i=0}^{n-j-3}\binom{n-j}{i} 2^{n-j-i-1}(2m+1)^{n-j-i-1}(-1)^i \notag \\
&\quad + a_j(-1)^{n-j-2}\binom{n-j}{n-j-2}(4m+2) + a_j(-1)^{n-j-1}(n-j) \notag\\
&\quad + a_{n-2}((4m+2)-2) + a_{n-1} \notag\\[4pt]
&= 4\Bigg(\sum_{j=0}^{n-3} a_j \sum_{i=0}^{n-j-3} a_j\binom{n-j}{i} 2^{n-j-i-3}(2m+1)^{n-j-i-1}(-1)^i \notag\\
&\qquad + a_j(-1)^{n-j-2}\binom{n-j}{n-j-2}(m) + a_{n-2}(m)\Bigg) \notag\\
&\quad + \sum_{j=0}^{n-3} (-1)^{n-j-2} a_j (n-j)(n-j-2) + a_{n-1} \tag{2.5}
\end{align}
Now we arrange the terms in the equation such that all terms divisible by 4 are on one side, and the terms not divisible by 4 are on the other side. Let $W$ be equal to
\begin{equation}
W = \sum_{j=0}^{n-3} a_j \sum_{i=0}^{n-j-3} a_j\binom{n-j}{i} 2^{n-j-i-3}(2m+1)^{n-j-i-1}(-1)^i + a_j(-1)^{n-j-2}\binom{n-j}{n-j-2}(m)+a_{n-2}(m). \tag{2.6}
\end{equation}
From equations (2.5) and (2.6), we get
\begin{equation}
f(4m+1) = 4W + \sum_{j=0}^{n-3}(-1)^{n-j-2}a_j(n-j)(n-j-2)+a_{n-1}. \tag{2.7}
\end{equation}
Now assume that
\begin{equation}
\sum_{j=0}^{n-3}(-1)^{n-j-2}a_j(n-j)(n-j-2)+a_{n-1} = 4K-1 \quad \text{where } K\ge 0\in\mathbb{Z}. \tag{2.8}
\end{equation}

It is assumed in equation (2.8), and from equation (2.7) we obtain that
\[
f(4m+1) = 4D \pm 4K - 1 = 4(D\pm K)-1 = 4B-1.
\]
Note that $4D+4(\pm K)-1 = 4B-1$ because $D>K$, so $K\in\mathbb{Z}$, and we have $f(4m+1)=4B-1$. Then from equation (2.1) we find that
\[
M(y^2+\alpha^2) = (4m+2)(4B-1).
\]
So, according to THEOREM~\ref{thm2.1}, there exists a prime number of the form $q=4n-1$ that divides $4B-1$. Therefore $f(4m+1)=4B-1=qw$ where $q=4n-1$; hence
\[
M(y^2+\alpha^2) = (4m+2)qw.
\]
We know all the divisors of the number $M$ are of the form $p=4n+1$, so we find that $M\not\equiv 0\pmod q$, then $y^2+\alpha^2\equiv 0\pmod q$. But according to Fermat's theorem on the sum of two squares, THEOREM~\ref{thm2.2}, all prime divisors of $y^2+\alpha^2$ are of the form $p=4n+1$. From this, we deduce
\[
M(y^2+\alpha^2) \ne (4m+2)qw.
\]
Then
\begin{equation}
M(y^2+\alpha^2) \ne (x+1)f(x) \quad \text{if } (x=4m+1,y)\in\mathbb{Z}^+\times\mathbb{Z}, \tag{2.9}
\end{equation}
and
\[
\sum_{j=0}^{n-3}(-1)^{n-j-2}a_j(n-j)(n-j-2)+a_{n-1} = 4K-1.
\]

Now assume $(x=4m,y)\in\mathbb{Z}^+\times\mathbb{Z}$, so we have that
\[
M(y^2+\alpha^2) = (4m+1)f(4m),
\]
and
\begin{equation}
f(4m) = \sum_{j=0}^{n-1} a_j \sum_{i=0}^{n-j-1}\binom{n-j}{i}(4m+1)^{n-j-i-1}(-1)^i. \tag{2.10}
\end{equation}
Notice that if we expand the term $(4m+1)^{n-j-i-1}(-1)^i$, we get
\begin{align}
(4m+1)^{n-j-i-1}(-1)^i &= 4(-1)^i\left(\sum_{k=0}^{n-j-i-2}\binom{n-j-i-1}{k}4^{n-j-i-2-k}m^{n-j-i-1-k}\right) \notag\\
&\quad + (-1)^i \tag{2.11}
\end{align}
Let $W^i_j$ equal the following:
\begin{equation}
W^i_j = \sum_{k=0}^{n-j-i-2}(-1)^i\binom{n-j-i-1}{k}4^{n-j-i-2-k}m^{n-j-i-1-k}. \tag{2.12}
\end{equation}
So from equations (2.11) and (2.12), we find that
\begin{equation}
(4m+1)^{n-j-i-1}(-1)^i = 4W^i_j + (-1)^i. \tag{2.13}
\end{equation}
By substituting equation (2.13) into equation (2.10), we obtain
\[
f(4m) = \sum_{j=0}^{n-1} a_j\sum_{i=0}^{n-j-1}\binom{n-j}{i}\big(4W^i_j+(-1)^i\big).
\]
Then
\begin{equation}
f(4m) = \sum_{j=0}^{n-1} a_j\sum_{i=0}^{n-j-1}\binom{n-j}{i}4W^i_j + \sum_{j=0}^{n-1} a_j\sum_{i=0}^{n-j-1}\binom{n-j}{i}(-1)^i. \tag{2.14}
\end{equation}
Note that
\[
\sum_{i=0}^{n-j-1}\binom{n-j}{i}(-1)^i = (1-1)^{n-j} - (-1)^{n-j}.
\]
Then
\begin{equation}
\sum_{i=0}^{n-j-1}\binom{n-j}{i}(-1)^i = (-1)^{n-j+1}. \tag{2.15}
\end{equation}
So from equations (2.14) and (2.15) we find that
\[
f(4m) = 4\left(\sum_{j=0}^{n-1} a_j\sum_{i=0}^{n-j-1}\binom{n-j}{i}W^i_j\right) + \sum_{j=0}^{n-1} a_j(-1)^{n-j+1}.
\]
Let $d$ be equal to
\[
d = \sum_{j=0}^{n-1} a_j\sum_{i=0}^{n-j-1}\binom{n-j}{i}W^i_j.
\]
Then
\begin{equation}
f(4m) = 4d + \sum_{j=0}^{n-1} a_j(-1)^{n-j+1}. \tag{2.16}
\end{equation}
Now assume that
\begin{equation}
\sum_{j=0}^{n-1} a_j(-1)^{n-j+1} = 4U-1 \quad \text{where } U\ge 0\in\mathbb{Z}. \tag{2.17}
\end{equation}
Now, from the assumption in (2.17) and the equation in (2.16), we find that
\[
f(4m) = 4d + 4(\pm U) - 1 = 4D-1 \quad \text{if } D = d\pm U \text{ because } d>U.
\]
Then
\[
M(y^2+\alpha^2) = (4m+1)(4D-1).
\]
So, according to THEOREM~\ref{thm2.1}, there exists a prime number of the form $q=4n-1$ that divides $4D-1$. Therefore $f(4m)=4D-1=qw$ where $q=4n-1$, hence
\[
M(y^2+\alpha^2) = (4m+1)(qw).
\]
But according to Fermat's theorem on the sum of two squares, THEOREM~\ref{thm2.2}, all prime divisors of $y^2+\alpha^2$ are of the form $q=4n+1$. From this, we deduce
\[
M(y^2+\alpha^2) \ne (4m+1)(qw) \quad \text{if } \sum_{j=0}^{n-1} a_j(-1)^{n-j+1} = 4U-1.
\]
Then
\begin{equation}
M(y^2+\alpha^2) \ne (x+1)f(x) \quad \text{if } (x=4m,y)\in\mathbb{Z}^+\times\mathbb{Z}. \tag{2.18}
\end{equation}

Now, according to equations (2.3), (2.9), (2.18), we obtain that
\[
M(y^2+\alpha^2) \ne (x+1)f(x)
\]
if $(x=4m-1,4m+1,4m,4m+2,\,y) \in \mathbb{Z}^+\times\mathbb{Z}$ where $x>1,\,y\ge 1$. Note that all integers are written in one of these forms. Notice if $x>1\in\mathbb{Z}^+$, this means $x=4m-1$ or $4m+1$ or $4m$ or $4m+2$. We deduce from that:
\begin{equation}
(x=4m-1,4m+1,4m,4m+2,\,y)\in\mathbb{Z}^+\times\mathbb{Z} \;=\; (x,y)\in\mathbb{Z}^+\times\mathbb{Z}, \quad \text{let } x>1,\, y\ge 1. \tag{2.19}
\end{equation}
We now deduce from all those equations (2.3), (2.9), (2.18), (2.19) the following result, provided that the conditions imposed on the coefficients are met in order to reach the final proof. If
\[
f(x) = \sum_{j=0}^{n-1}\sum_{i=0}^{n-j-1} a_j\binom{n-j}{i}(x+1)^{n-j-i-1}(-1)^i,
\]
\[
\sum_{j=0}^{n-1} a_j(-1)^{n-j+1} = 4k-1, \qquad \sum_{j=0}^{n-3}(-1)^{n-j-2}a_j(n-j)(n-j-2)+a_{n-1} = 4U-1,
\]
\[
U,K\ge 0\in\mathbb{Z}, \qquad \sum_{j=0}^{n}(-1)^{n-j+1}a_j = M\alpha^2,
\]
then
\[
M(y^2+\alpha^2) \ne (x+1)f(x) \quad \text{for all } (x,y)\in\mathbb{Z}^+\times\mathbb{Z} \text{ where } x>1 \text{ and } y\ge 1. \qedhere
\]
\end{proof}

\section{Solutions of Algebraic Curves}

In this section, we study the relation between the solutions of algebraic curves and functions of the form $M(y^2+\alpha^2)=(x+1)f(x)$. We find that the integral solutions of algebraic curves are connected to this class of functions. This is explained in THEOREM~\ref{thm3.1}. Subsequently, we prove THEOREM~\ref{thm3.2}. The proof relies on THEOREM~\ref{thm3.1} and THEOREM~\ref{thm2.3}. THEOREM~\ref{thm3.2} establishes the non-existence of positive integral solutions for any algebraic curve of degree $n$, provided that certain algebraic conditions on the coefficients of the algebraic curve are satisfied.

\begin{theorem}\label{thm3.1}
Let $C: My^2 = a_0 x^n + a_1 x^{n-1} + a_2 x^{n-2} + \cdots + a_n$ be an algebraic curve, nonsingular of genus $g\ge 0$ and $\deg C \ge 2$, where $a_0,a_1,\dots,a_n\in\mathbb{Z}$ are the coefficients of the curve $C$, where $\alpha,M\ge 1$ are odd numbers, all the prime divisors of which are of the form $p=4n+1$, and $C(\mathbb{Z}^+)$ are positive integer points. Then we have
\[
\text{if } M(y^2+\alpha^2) \ne (x+1)f(x) \text{ then } (x,y)\notin C(\mathbb{Z}^+),
\]
where
\[
f(x) = \sum_{j=0}^{n-1}\sum_{i=0}^{n-j-1} a_j\binom{n-j}{i}(x+1)^{n-j-i-1}(-1)^i,
\]
\[
\sum_{j=0}^{n}(-1)^{n-j+1}a_j = M\alpha^2, \qquad \sum_{j=0}^{n-1} a_j(-1)^{n-j+1} = 4k-1,
\]
\[
\sum_{j=0}^{n-3}(-1)^{n-j-2}a_j(n-j)(n-j-2)+a_{n-1} = 4U-1, \qquad U,K\ge 0\in\mathbb{Z}.
\]
\end{theorem}

\begin{proof}
Let $C: My^2 = a_0 x^n + a_1 x^{n-1} + \cdots + a_{n-1}x + a_n$ be an algebraic curve, nonsingular of genus $g\ge 0$, with coefficients $a_0,a_1,a_2,\dots,a_n\in\mathbb{Z}$, where $M,\alpha\ge 1$ are odd constants, all the prime divisors of $M,\alpha$ are of the form $p=4m+1$, and $C(\mathbb{Z}^+)$ are the positive integer solutions for curve $C$, and the coefficients of the curve possess the following property:
\[
\sum_{j=0}^{n-1} a_j(-1)^{n-j+1} = 4k-1, \qquad \sum_{j=0}^{n-3}(-1)^{n-j-2}a_j(n-j)(n-j-2)+a_{n-1}=4U-1, \qquad U,K\ge 0\in\mathbb{Z},
\]
\begin{equation}
M\alpha^2 = -\sum_{j=0}^{n}(-1)^{n-j}a_j. \tag{3.1}
\end{equation}
Now, suppose we have the following equation
\begin{equation}
M(y^2+\alpha^2) \ne (x+1)\sum_{j=0}^{n-1}\sum_{i=0}^{n-j-1} a_j\binom{n-j}{i}(x+1)^{n-j-i-1}(-1)^i \quad \text{for all } (x,y)\in\mathbb{Z}^+\times\mathbb{Z}. \tag{3.2}
\end{equation}
So we have that
\begin{equation}
My^2 \ne \sum_{j=0}^{n-1}\sum_{i=0}^{n-j-1} a_j\binom{n-j}{i}(x+1)^{n-j-i}(-1)^i - M\alpha^2. \tag{3.3}
\end{equation}
Notice from equations (3.1) and (3.3) we obtain that
\begin{equation}
My^2 \ne \sum_{j=0}^{n-1}\sum_{i=0}^{n-j-1} a_j\binom{n-j}{i}(x+1)^{n-j-i}(-1)^i + \sum_{j=0}^{n} (-1)^{n-j} a_j. \tag{3.4}
\end{equation}
Notice that
\begin{equation}
\sum_{i=0}^{n-j-1} a_j\binom{n-j}{i}(x+1)^{n-j-i}(-1)^i + (-1)^{n-j}a_j = a_j\big((x+1)-1\big)^{n-j}. \tag{3.5}
\end{equation}
By substituting equation (3.5) into equation (3.4), we find that
\[
My^2 \ne \sum_{j=0}^{n} a_j\big((x+1)-1\big)^{n-j}.
\]
Let $1,-1\in G=(\mathbb{Z},+)$ be a group and $e\in G$ be the identity element. Note that $1-1=e$, thus we get
\[
My^2 \ne \sum_{j=0}^{n-1} a_j(x+e)^{n-j}.
\]
So we have that
\begin{equation}
My^2 \ne \sum_{j=0}^{n} a_j x^{n-j} \quad \text{this means } (x,y)\notin C(\mathbb{Z}^+). \tag{3.6}
\end{equation}
Therefore, from equations (3.6) and (3.2) we have that
\[
\text{if } M(x^2+\alpha^2) \ne (x+1)f(x) \text{ then } (x,y)\notin C(\mathbb{Z}^+),
\]
where
\[
f(x) = \sum_{j=0}^{n-1}\sum_{i=0}^{n-j-1} a_j\binom{n-j}{i}(x+1)^{n-j-i-1}(-1)^i, \qquad M\alpha^2 = \sum_{j=0}^{n}(-1)^{n-j+1}a_j,
\]
\[
\sum_{j=0}^{n-1} a_j(-1)^{n-j+1} = 4k-1, \qquad \sum_{j=0}^{n-3}(-1)^{n-j-2}a_j(n-j)(n-j-2)+a_{n-1}=4U-1, \qquad U,K\ge0\in\mathbb{Z}. \qedhere
\]
\end{proof}

\begin{theorem}\label{thm3.2}
Let $C: My^2 = a_0 x^n + a_1 x^{n-1} + \cdots + a_{n-1}x + a_n$ be an algebraic curve, nonsingular of genus $g\ge 0$ where $\deg C\ge 2$, with coefficients $a_0,a_1,a_2,\dots,a_n\in\mathbb{Z}$, where $M,\alpha\ge 1$ are odd constants, all the prime divisors of $M,\alpha$ are of the form $p=4m+1$, and $C(\mathbb{Z}^+)$ are the positive integer solutions for curve $C$. Then, if
\[
\sum_{j=0}^{n}(-1)^{n-j+1}a_j = M\alpha^2, \qquad \sum_{j=0}^{n-1} a_j(-1)^{n-j+1} = 4k-1,
\]
\[
\sum_{j=0}^{n-3}(-1)^{n-j-2}a_j(n-j)(n-j-2)+a_{n-1} = 4U-1, \qquad U,K\ge 0\in\mathbb{Z},
\]
and if $x>1$ and $y\ge 1$ where $\forall (x,y)\in C(\mathbb{Z}^+)$ holds, then the curve $C$ admits no solutions in positive integers. That is,
\[
C(\mathbb{Z}^+) = \emptyset.
\]
\end{theorem}

\begin{proof}
According to THEOREM~\ref{thm3.1}, let $C: My^2 = a_0 x^n + a_1 x^{n-1} + \cdots + a_n$ be an algebraic curve, nonsingular of genus $g\ge 0$, where $a_0,a_1,\dots,a_n\in\mathbb{Z}$ are the coefficients of the curve $C$, where $\alpha,M$ are odd numbers, all the prime divisors of which are of the form $p=4n+1$, and $C(\mathbb{Z}^+)$ are positive integer points. Then
\begin{equation}
\text{if } M(y^2+\alpha^2) \ne (x+1)f(x) \text{ then } (x,y)\notin C(\mathbb{Z}^+), \tag{3.7}
\end{equation}
where
\[
\sum_{j=0}^{n}(-1)^{n-j+1}a_j = M\alpha^2.
\]
\[
\sum_{j=0}^{n-1} a_j(-1)^{n-j+1} = 4k-1, \qquad \sum_{j=0}^{n-3}(-1)^{n-j-2}a_j(n-j)(n-j-2)+a_{n-1}=4U-1, \qquad U,K\ge0\in\mathbb{Z}.
\]
And according to THEOREM~\ref{thm2.3}, let $a_0,a_1,\dots,a_n,\alpha,M\in\mathbb{Z}$ be integers and $M,\alpha\ge 1$ be odd numbers, all prime divisors of $M,\alpha = p_1\cdot p_2\cdot p_3\cdots p_n$ of the form $p_j=4n+1$. Then we have
\begin{equation}
M(y^2+\alpha^2) \ne (x+1)f(x) \quad \text{for all } (x,y)\in(\mathbb{Z}^2)^+,\ x>1 \text{ and } y\ge 1, \tag{3.8}
\end{equation}
if
\[
\sum_{j=0}^{n}(-1)^{n-j+1}a_j = M\alpha^2, \qquad \sum_{j=0}^{n-1} a_j(-1)^{n-j+1} = 4k-1,
\]
\[
\sum_{j=0}^{n-3}(-1)^{n-j-2}a_j(n-j)(n-j-2)+a_{n-1} = 4U-1, \qquad U,K\ge 0\in\mathbb{Z}.
\]
We know $C(\mathbb{Z}^+) \subset (\mathbb{Z}^2)^+$; therefore, we conclude from equations (3.7) and (3.8) that if
\[
\sum_{j=0}^{n}(-1)^{n-j+1}a_j = M\alpha^2, \qquad \sum_{j=0}^{n-1} a_j(-1)^{n-j+1} = 4k-1,
\]
\[
\sum_{j=0}^{n-3}(-1)^{n-j-2}a_j(n-j)(n-j-2)+a_{n-1} = 4U-1, \qquad U,K\ge 0\in\mathbb{Z},
\]
then
\begin{equation}
M(y^2+\alpha^2) \ne (x+1)f(x) \quad \text{for all } (x,y)\in C(\mathbb{Z}^+),\ x>1 \text{ and } y\ge 1. \tag{3.9}
\end{equation}
Now from equations (3.9) and (3.7), this implies that
\[
(x,y)\notin C(\mathbb{Z}^+) \quad \text{for each } (x,y) > (1,0) \in (\mathbb{Z}^2)^+.
\]
Then
\begin{equation}
C(\mathbb{Z}^+) = \emptyset \quad \text{if } x>1 \text{ and } y\ge 1. \tag{3.10}
\end{equation}
Therefore, $C(\mathbb{Z}^+)=\emptyset$ if $(x,y)\in C(\mathbb{Z}^+)$ where $x>1,y\ge 1$, with the existence of the conditions on the coefficients. Then there is no solution for curve $C$ in the set of positive integers.

Thus, by (3.7), (3.9), (3.10), we conclude that if certain conditions are satisfied by the coefficients of the algebraic curve, then the curve has no solution if
\[
\sum_{j=0}^{n}(-1)^{n-j+1}a_j = M\alpha^2.
\]
\[
\sum_{j=0}^{n-1} a_j(-1)^{n-j+1} = 4k-1, \qquad \sum_{j=0}^{n-3}(-1)^{n-j-2}a_j(n-j)(n-j-2)+a_{n-1} = 4U-1, \qquad U,K\ge 0\in\mathbb{Z}.
\]
Then
\[
C(\mathbb{Z}^+) = \emptyset. \qedhere
\]
\end{proof}

\begin{theorem}\label{thm3.3}
Let $C: y^2 = a_0 x^n + a_1 x^{n-1} + \cdots + a_{n-1}x + a_n$ be an algebraic curve, nonsingular of genus $g\ge 0$, where $\deg C\ge 2$, with coefficients $a_0,a_1,a_2,\dots,a_n\in\mathbb{Z}$, and $C(\mathbb{Z}^+)$ are the positive integer solutions for curve $C$. Then, if
\[
\sum_{j=0}^{n}(-1)^{n-j+1}a_j = 1, \qquad \sum_{j=0}^{n-1} a_j(-1)^{n-j+1} = 4k-1,
\]
\[
\sum_{j=0}^{n-3}(-1)^{n-j-2}a_j(n-j)(n-j-2)+a_{n-1} = 4U-1, \qquad U,K\ge 0\in\mathbb{Z},
\]
and if $x>1$ and $y\ge 1$ where $\forall (x,y)\in C(\mathbb{Z}^+)$ holds. Consequently, the curve $C$ admits no solutions in positive integers. That is,
\[
C(\mathbb{Z}^+) = \emptyset.
\]
\end{theorem}

\begin{proof}
Let $K=M=\alpha=1$ in THEOREM~\ref{thm3.2}.
\end{proof}

\section{The Relationship Between Linear Algebra and Algebraic Geometry}

In this section, we investigate the relationship between solutions of systems of linear equations and solutions of algebraic curves. By THEOREM~\ref{thm3.2}, we prove that if a given linear system $S: Ax=b$ has a solution, then one can construct an algebraic curve whose coefficients encode precisely the solutions of that system, while the resulting curve has no positive integer solutions. We identify the conditions that the system $S$ must satisfy for such a curve to be constructed; these conditions are formulated in THEOREM 4.1. In addition, we explain how to construct distinct systems of linear equations, as described in THEOREM 4.2 and 4.3. We also construct linear systems with different prescribed properties, reflected in the entries of the vector $b$; the corresponding construction is presented in THEOREM 4.4. Finally, we show that if a linear system has infinitely many solutions, then there exist infinitely many algebraic curves having no positive integer solutions.

\begin{theorem}\label{thm4.1}
Let $S: Ax=b$ be a system of linear equations over $\mathbb{Z}$, and $A\in M_{3\times(n+1)}(\mathbb{Z})$. And $C_{s_i}\in\mathbb{Z}(x,y)$ be plane algebraic curves, and $M,k\ge 1\in\mathbb{N}$, $U,D\ge 0\in\mathbb{Z}$ where $M,k\ge 1$ are odd positive integers and every prime divisor of $M,k$ satisfies $p\equiv 1\pmod 4$. If the linear system $S$ is
\[
\sbmat{
(-1)^{n-2}a_{11}\epsilon_0 & (-1)^{n-3}a_{12}\epsilon_1 & \cdots & (-1)^{1}a_{1,n-2}\epsilon_{n-3} & 0 & a_{1n} & 0 \\
(-1)^{n+1}a_{21} & (-1)^{n}a_{22} & \cdots & (-1)^{3}a_{2,n-1} & (-1)^{2}a_{2n} & 0 & \\
(-1)^{n-1}a_{31} & (-1)^{n-2}a_{32} & \cdots & (-1)^{3}a_{3,n-1} & (-1)^{2}a_{3n} & (-1)^{1}a_{3,n+1} &
}
\sbmat{c_1\\c_2\\c_3\\ \vdots \\ \vdots \\ c_{n+1}}
=
\sbmat{4U-1\\4D-1\\Mk^2}
\]
Where $\epsilon_j = (n-j)(n-j-2)$, and any three elements in any column are equal, $a_{1j}=a_{2j}=a_{3j}$, and $n$ denotes the degree of the curve $C$, $\deg C = n\ge 2$. If there is a solution to the system of linear equations $S$, and the solution is $S=\{s_1,s_2,\dots,s_N\}\subset \mathbb{Z}^{n+1}$ where $s_i=(c_1,c_2,\dots,c_{n+1})$, then the following curve exists, nonsingular curves $C_{s_i}$ of genus $g\ge 0$, where
\[
C_{s_i}: My^2 = a_{31}c_1 x^n + a_{32}c_2 x^{n-1} + a_{33}c_3 x^{n-2} + \cdots + a_{3,n+1}c_{n+1}.
\]
If $y\ge 1$, $x>1$ and $\forall (x,y)\in C_{s_i}(\mathbb{Z}^+)$ holds then
\[
\big\{C_{s_i}(\mathbb{Z}^+)\big\}^{N}_{i=1} = \emptyset.
\]
\end{theorem}

\begin{proof}
According to THEOREM~\ref{thm3.2}, let $C: My^2 = a_0 x^n + a_1 x^{n-1} + a_2 x^{n-2} + \cdots + a_n$ be an algebraic curve, nonsingular of genus $g\ge 0$, where $a_0,a_1,\dots,a_n\in\mathbb{Z}$ are the coefficients of the curve, $\deg C = n\ge 2$, where $k,M$ are odd numbers, all the prime divisors of which are of the form $p=4m+1$, and $C(\mathbb{Z}^+)$ are positive integer points. Then, if
\[
\sum_{j=0}^{n}(-1)^{n-j+1}a_j = Mk^2, \qquad \sum_{j=0}^{n-1} a_j(-1)^{n-j+1} = 4D-1,
\]
\[
\sum_{j=0}^{n-3}(-1)^{n-j-2}a_j(n-j)(n-j-2)+a_{n-1} = 4U-1, \qquad U,D\ge 0\in\mathbb{Z},
\]
then there are no solutions in the set of positive integers for the curve $C$; in other words, $C(\mathbb{Z}^+)=\emptyset$ if $x>1$, $y\ge 1$ where $\forall (x,y)\in C(\mathbb{Z}^+)$.

Now, let $n$ be the degree of the curve $C$. Furthermore, let $a_0 = a_{11}c_1$ where $a_{11},c_1\in\mathbb{Z}$, and, similarly, let $a_1=a_{12}c_2$ where $a_{12},c_2\in\mathbb{Z}$. For all coefficients of the curve $C$ where $a_j = a_{1,j+1}c_{j+1}$, $j=1,2,3,\dots,n$ and $a_{1,j+1},c_{j+1}\in\mathbb{Z}$, we deduce the following:
\[
\sum_{j=0}^{n-3}(-1)^{n-j-2}a_{1,j+1}c_{j+1}(n-j)(n-j-2) + a_{1n}c_n = 4U-1,
\]
\[
\sum_{j=0}^{n}(-1)^{n-j+1}a_{3,j+1}c_{j+1} = Mk^2, \qquad \sum_{j=0}^{n-1} a_{2,j+1}c_{j+1}(-1)^{n-j+1} = 4D-1,
\]
where $a_{1,j+1}c_{j+1} = a_{2,j+1}c_{j+1} = a_{3,j+1}c_{j+1}$, and likewise,
\[
C: My^2 = a_{31}c_1 x^n + a_{32}c_2 x^{n-1} + a_{33}c_3 x^{n-2} + \cdots + a_{n+1}c_{n+1}.
\]

Now, we arrange all the coefficients of the curve $C$ into a matrix system $S: Ax=b$ where $A\in M_{3\times(n+1)}(\mathbb{Z})$, and thus we find
\[
\sbmat{
(-1)^{n-2}a_{11}\epsilon_0 & (-1)^{n-3}a_{12}\epsilon_1 & \cdots & (-1)^{1}a_{1,n-2}\epsilon_{n-3} & 0 & a_{1n} & 0 \\
(-1)^{n+1}a_{21} & (-1)^{n}a_{22} & \cdots & (-1)^{3}a_{2,n-1} & (-1)^{2}a_{2n} & 0 & \\
(-1)^{n-1}a_{31} & (-1)^{n-2}a_{32} & \cdots & (-1)^{3}a_{3,n-1} & (-1)^{2}a_{3n} & (-1)^{1}a_{3,n+1} &
}
\sbmat{c_1\\c_2\\c_3\\ \vdots \\ \vdots \\ c_{n+1}}
=
\sbmat{4U-1\\4D-1\\Mk^2},
\]
where $\epsilon_j=(n-j)(n-j-2)$, and any three elements in any column are equal, $a_{1j}=a_{2j}=a_{3j}$, and $n$ denotes the degree of the curve $C$, $\deg C=n$. This means that if there is a solution to the system of linear equations $S$, and the solution is $s_1=(c_1,c_2,\dots,c_{n+1})\subset\mathbb{Z}^{n+1}$, then the following curve exists:
\[
C_{s_1}: My^2 = a_{31}c_1x^n + a_{32}c_2x^{n-1} + a_{33}c_3x^{n-2} + \cdots + a_{3,n+1}c_{n+1}.
\]
Then, by THEOREM~\ref{thm3.2}, the coefficients of the curve satisfy the criterion. From the conditions, we deduce that if $y\ge 1$, $x>1$ and $\forall(x,y)\in C_{s_1}(\mathbb{Z}^+)$ holds, then
\[
C_{s_1}(\mathbb{Z}^+) = \emptyset.
\]
Now assume that there exists another solution to the system with $s_2=(c_1,c_2,\dots,c_{n+1})$. Consequently, we obtain a new curve $C_{s_2}$ where
\[
C_{s_2}: My^2 = a_{31}c_1x^n + a_{32}c_2x^{n-1} + a_{33}c_3x^{n-2}+\cdots+a_{3,n+1}c_{n+1}.
\]
Then the coefficients arising from this solution satisfy the hypotheses of THEOREM~\ref{thm3.2}. Hence we deduce that if $y\ge 1$, $x>1$ and $\forall(x,y)\in C_{s_2}(\mathbb{Z}^+)$ holds then
\[
C_{s_2}(\mathbb{Z}^+) = \emptyset.
\]
Accordingly, if the system $S$ admits $N$ solutions with $S=\{s_1,s_2,\dots,s_N\}\subset\mathbb{Z}^{n+1}$, this implies the existence of $N$ algebraic curves $C_{s_i}$; if $y\ge 1$, $x>1$ and $\forall(x,y)\in C_{s_i}(\mathbb{Z}^+)$ holds then
\[
\big\{C_{s_i}(\mathbb{Z}^+)\big\}^{N}_{i=1} = \emptyset. \qedhere
\]
\end{proof}

\begin{lemma}\label{lem4.1}
Let $S: Ax=b$ be a system of linear equations over $\mathbb{Z}$, and $A\in M_{3\times(n+1)}(\mathbb{Z})$. And $C_{s_i}\in\mathbb{Z}(x,y)$ be plane algebraic curves, and $K\ge 1\in\mathbb{N}$, $U,D\ge 0\in\mathbb{Z}$, where $K$ is an odd positive integer and every prime divisor of $K$ satisfies $p\equiv 1\pmod 4$. If the linear system $S$ is
\[
\sbmat{
(-1)^{n-2}a_{11}\epsilon_0 & (-1)^{n-3}a_{12}\epsilon_1 & \cdots & (-1)^{1}a_{1,n-2}\epsilon_{n-3} & 0 & a_{1n} & 0 \\
(-1)^{n+1}a_{21} & (-1)^{n}a_{22} & \cdots & (-1)^{3}a_{2,n-1} & (-1)^{2}a_{2n} & 0 & \\
(-1)^{n-1}a_{31} & (-1)^{n-2}a_{32} & \cdots & (-1)^{3}a_{3,n-1} & (-1)^{2}a_{3n} & (-1)^{1}a_{3,n+1} &
}
\sbmat{c_1\\c_2\\c_3\\ \vdots \\ \vdots \\ c_{n+1}}
=
\sbmat{4U-1\\4D-1\\K^2},
\]
where $\epsilon_j=(n-j)(n-j-2)$, and any three elements in any column are equal, $a_{1j}=a_{2j}=a_{3j}$, and $n$ denotes the degree of the curve $C$, $\deg C=n\ge 2$. If there is a solution to the system of linear equations $S$, and the solution is $S=\{s_1,s_2,\dots,s_N\}\subset\mathbb{Z}^{n+1}$
where $s_i=(c_1,c_2,\dots,c_{n+1})$, then the following curve exists, nonsingular curves $C_{s_i}$ of genus $g\ge 0$, where
\[
C_{s_i}: y^2 = a_{31}c_1x^n + a_{32}c_2x^{n-1} + a_{33}c_3x^{n-2}+\cdots+a_{3,n+1}c_{n+1}.
\]
If $y\ge 1$, $x>1$ and $\forall(x,y)\in C_{s_i}(\mathbb{Z}^+)$ holds then
\[
\big\{C_{s_i}(\mathbb{Z}^+)\big\}^{N}_{i=1} = \emptyset.
\]
\end{lemma}

\begin{proof}
Let $M=1$ in THEOREM~\ref{thm4.1}.
\end{proof}

\begin{theorem}\label{thm4.2}
Let $S: Ax=b$ be a system of linear equations over $\mathbb{Z}$ and $A\in M_{6\times(2m+2)}(\mathbb{Z})$. Let $C_{s_i},G_{s_i}\in\mathbb{Z}(x,y)$ be plane algebraic curves, nonsingular curves of genus $g\ge 0$, and $v,u,d,t\ge 0\in\mathbb{Z}$ and $M,a,D,k\ge 1$ be positive odd integers such that every prime divisor of each satisfies $p\equiv 1\pmod 4$. If the linear system $S$ is
\[
\sbmat{
-a_{11}\epsilon_0 & a_{12}\epsilon_1 & -a_{13}\epsilon_2 & \cdots & -a_{1,2m-1}\epsilon_{2m-2} & 0 & a_{1,2m+1} & 0\\
a_{21} & a_{22} & -a_{23} & \cdots & -a_{2,2m} & a_{2,2m+1} & 0 & \\
a_{31} & -a_{32} & a_{33} & \cdots & a_{3,2m-1} & -a_{3,2m} & a_{3,2m+1} & -a_{3,2m+2}\\
a_{41}\varepsilon_0 & -a_{42}\varepsilon_1 & a_{43}\varepsilon_2 & \cdots & -a_{4,2m-2}\varepsilon_{2m-3} & 0 & a_{4,2m} & 0 & 0\\
-a_{51} & a_{52} & -a_{53} & \cdots & -a_{5,2m-1} & a_{5,2m} & 0 & 0\\
-a_{61} & a_{62} & -a_{63} & \cdots & -a_{6,2m-1} & a_{6,2m} & -a_{6,2m+1} & 0
}
\sbmat{c_1\\c_2\\c_3\\ \vdots\\ \vdots\\ \vdots\\ c_{2m+2}}
=
\sbmat{4d-1\\4v-1\\Ma^2\\4t-1\\4u-1\\Dk^2}
\]
where $\epsilon_j = (2m+1-j)(2m-j-1)$ and $\varepsilon_j = (2m-j)(2m-j-2)$, and there exist two algebraic curves of degree $\deg C_{s_i}=2m$ and $\deg G_{s_i}=2m+1$, and every three consecutive elements in any column are equal, $a_{1j}=a_{2j}=a_{3j}$, and $a_{4j}=a_{5j}=a_{6j}$. Let $\{s_1,s_2,s_3,\dots,s_N\}\subset\mathbb{Z}^{2m+2}$ denote the set of integer solutions of the system of linear equations $S$. Then for every solution $s_i=(c_1,c_2,\dots,c_{2m+2})$ there exist two algebraic curves $C_{s_i}$ and $G_{s_i}$ where
\[
G_{s_i}: MY^2 = a_{31}c_1X^{2m+1} + a_{32}c_2X^{2m} + a_{33}c_3X^{2m-1} + \cdots + a_{3,2m+2}c_{2m+2},
\]
\[
C_{s_i}: DY^2 = a_{61}c_1X^{2m} + a_{62}c_2X^{2m-1} + a_{63}c_3X^{2m-2} + \cdots + a_{6,2m+1}c_{2m+1}.
\]
If $y\ge 1$, $x>1$ and $\forall(x,y)\in C_{s_i}(\mathbb{Z}^+)$ and $\forall(x,y)\in G_{s_i}(\mathbb{Z}^+)$ holds then
\[
\big\{C_{s_i}(\mathbb{Z}^+)\big\}^{N}_{i=1} = \emptyset, \qquad \big\{G_{s_i}(\mathbb{Z}^+)\big\}^{N}_{i=1} = \emptyset.
\]
\end{theorem}

\begin{proof}
By THEOREM~\ref{thm4.1}, let $S: Ax=b$ be a system of linear equations over $\mathbb{Z}$, and $A\in M_{3\times(n+1)}(\mathbb{Z})$. And $C_{s_i}\in\mathbb{Z}(x,y)$ be plane algebraic curves, nonsingular, genus $g\ge 0$ with integer coefficients, and $M,k\ge 1\in\mathbb{N}$, $U,D\ge 0\in\mathbb{Z}$, where $M$ and $k$ are odd positive integers and every prime divisor of $M,k$ satisfies $p\equiv 1\pmod 4$. If the linear system $S$ is
\[
\sbmat{
(-1)^{n-2}a_{11}\epsilon_0 & (-1)^{n-3}a_{12}\epsilon_1 & \cdots & (-1)^{1}a_{1,n-2}\epsilon_{n-3} & 0 & a_{1n} & 0 \\
(-1)^{n+1}a_{21} & (-1)^{n}a_{22} & \cdots & (-1)^{3}a_{2,n-1} & (-1)^{2}a_{2n} & 0 & \\
(-1)^{n-1}a_{31} & (-1)^{n-2}a_{32} & \cdots & (-1)^{3}a_{3,n-1} & (-1)^{2}a_{3n} & (-1)^{1}a_{3,n+1} &
}
\sbmat{c_1\\c_2\\c_3\\ \vdots \\ \vdots \\ c_{n+1}}
=
\sbmat{4U-1\\4D-1\\MK^2},
\]
where $\epsilon_j=(n-j)(n-j-2)$, and any three elements in any column are equal, $a_{1j}=a_{2j}=a_{3j}$, and $n$ denotes the degree of the curve $C$, $\deg C=n$. If there is a solution to the system of linear equations $S$, and the solution is $S=\{s_1,s_2,\dots,s_N\}\subset\mathbb{Z}^{n+1}$ where $s_i=(c_1,c_2,\dots,c_{n+1})$, then the following curve exists:
\[
C_{s_i}: My^2 = a_{31}c_1x^n + a_{32}c_2x^{n-1} + a_{33}c_3x^{n-2}+\cdots+a_{3,n+1}c_{n+1}.
\]
If $y\ge 1$, $x>1$ and $\forall(x,y)\in C_{s_i}(\mathbb{Z}^+)$ holds then
\[
\big\{C_{s_i}(\mathbb{Z}^+)\big\}^{N}_{i=1} = \emptyset.
\]

Now let $C_1$ and $G_2 \in \mathbb{Z}(x,y)$ be two curves. Using the coefficients of the curves $C_1,G_2$, one can construct the linear systems $S: Ax=b$ and $K: Bx=d$, where $A\in M_{3\times(2m+2)}(\mathbb{Z})$ and $B\in M_{3\times 2m}(\mathbb{Z})$, where $\deg C=2m$ and $\deg G=2m+1$. Now suppose that $S$ and $K$ admit a common solution. Then the two systems may be amalgamated into a single linear system $F: Tx=q$ where $T\in M_{6\times(2m+1)}(\mathbb{Z})$. Consequently, we obtain the system $F$:
\[
\sbmat{
-a_{11}\epsilon_0 & a_{12}\epsilon_1 & -a_{13}\epsilon_2 & \cdots & -a_{1,2m-1}\epsilon_{2m-2} & 0 & a_{2m+1} & 0\\
a_{21} & a_{22} & -a_{23} & \cdots & -a_{2,2m} & a_{2,2m+1} & 0 & \\
a_{31} & -a_{32} & a_{33} & \cdots & a_{3,2m-1} & -a_{3,2m} & a_{3,2m+1} & -a_{3,2m+2}\\
a_{41}\varepsilon_0 & -a_{42}\varepsilon_1 & a_{43}\varepsilon_2 & \cdots & -a_{4,2m-2}\varepsilon_{2m-3} & 0 & a_{4,2m} & 0 & 0\\
-a_{51} & a_{52} & -a_{53} & \cdots & -a_{5,2m-1} & a_{5,2m} & 0 & 0\\
-a_{61} & a_{62} & -a_{63} & \cdots & -a_{6,2m-1} & a_{6,2m} & -a_{6,2m+1} & 0
}
\sbmat{c_1\\c_2\\c_3\\ \vdots\\ \vdots\\ \vdots\\ c_{2m+2}}
=
\sbmat{4d-1\\4v-1\\Ma^2\\4t-1\\4u-1\\Dk^2}
\]
where $\epsilon_j=(2m+1-j)(2m-j-1)$ and $\varepsilon_j=(2m-j)(2m-j-2)$, with the condition that any three consecutive elements in any column are equal, $a_{1j}=a_{2j}=a_{3j}$ and $a_{4j}=a_{5j}=a_{6j}$, where $d,v,t,u\ge 0\in\mathbb{Z}$ and $M,D,a,k\ge 1\in\mathbb{N}$. Moreover, $M,D,a,k$ are odd positive integers, and in addition every prime divisor of them satisfies $p\equiv 1\pmod 4$. There exist two curves of degrees $\deg C=2m$ and $\deg G=2m+1$. By THEOREM~\ref{thm4.1}, if the linear system of equations $F$ admits a solution, where $s_1\subset\mathbb{Z}^{2m+2}$ denote the
solution of the linear system $F$, and $s_1=(c_1,c_2,\dots,c_{2m+2})$, then there exist two algebraic curves of degree $\deg G_{s_i}=2m+1$ and $\deg C_{s_i}=2m$ where
\[
C_{s_1}: MY^2 = a_{31}c_1X^{2m} + a_{32}c_2X^{2m-1} + a_{33}c_3X^{2m-2}+\cdots+a_{3,2m+1}c_{2m+1},
\]
\[
G_{s_1}: DY^2 = a_{61}c_1X^{2m+1} + a_{62}c_2X^{2m} + a_{63}c_3X^{2m-1}+\cdots+a_{6,2m+2}c_{2m+2}.
\]
In the linear system $F$, we observe that the coefficients of the curves $C_{s_i}$ and $G_{s_i}$ satisfy the hypotheses of THEOREM~\ref{thm4.1}. Suppose that the system $F$ admits a solution, which we denote by $s_1=(c_1,c_2,\dots,c_{2m+2})$. We deduce from this that, if $y\ge 1$, $x>1$ and $\forall(x,y)\in C_{s_i}(\mathbb{Z}^+)$ and $\forall(x,y)\in G_{s_i}(\mathbb{Z}^+)$ holds then
\[
C_{s_1}(\mathbb{Z}^+) = G_{s_1}(\mathbb{Z}^+) = \emptyset.
\]
Suppose now that there exist $N$ solutions of the form $\{s_1,s_2,s_3,\dots,s_N\}\subset\mathbb{Z}^{2m+2}$ to the linear system $F$. Let $F=S$ and $T=A$, $B=q$. By THEOREM~\ref{thm4.1}, we may therefore construct $N$ curves whose coefficients all satisfy the hypotheses of THEOREM~\ref{thm3.2}. If $S=\{s_1,s_2,s_3,\dots,s_N\}\subset\mathbb{Z}^{2m+2}$ and $s_i=(c_1,c_2,\dots,c_{2m+2})\in S$, hence
\[
C_{s_i}: MY^2 = a_{31}c_1X^{2m} + a_{32}c_2X^{2m-1} + a_{33}c_3X^{2m-2}+\cdots+a_{3,2m+1}c_{2m+1},
\]
\[
G_{s_i}: DY^2 = a_{61}c_1X^{2m+1} + a_{62}c_2X^{2m} + a_{63}c_3X^{2m-1}+\cdots+a_{6,2m+2}c_{2m+2},
\]
if $y\ge 1$, $x>1$ and $\forall(x,y)\in C_{s_i}(\mathbb{Z}^+)$ and $\forall(x,y)\in G_{s_i}(\mathbb{Z}^+)$ holds then
\[
\big\{C_{s_i}(\mathbb{Z}^+)\big\}^{N}_{i=1} = \emptyset, \qquad \big\{G_{s_i}(\mathbb{Z}^+)\big\}^{N}_{i=1} = \emptyset. \qedhere
\]
\end{proof}

\begin{theorem}\label{thm4.3}
Let $S: Ax=b$ be a system of linear equations over $\mathbb{Z}$ and $A\in M_{9\times(2m+2)}(\mathbb{Z})$. Let $C_{s_i},G_{s_i},V_{s_i}\in\mathbb{Z}(x,y)$ be plane algebraic curves, nonsingular curves of genus $g\ge 0$, where $v,u,d,t,s,q\ge 0\in\mathbb{Z}$. If the linear system $S$ is
\[
\sbmat{
-a_{11}\epsilon_0 & a_{12}\epsilon_1 & -a_{13}\epsilon_2 & \cdots & -a_{1,2m-1}\epsilon_{2m-2} & 0 & a_{2m+1} & 0\\
a_{21} & a_{22} & -a_{23} & \cdots & -a_{2,2m} & a_{2,2m+1} & 0 & \\
a_{31} & -a_{32} & a_{33} & \cdots & a_{3,2m-1} & -a_{3,2m} & a_{3,2m+1} & -a_{3,2m+2}\\
a_{41}\varepsilon_0 & -a_{42}\varepsilon_1 & a_{43}\varepsilon_2 & \cdots & -a_{4,2m-2}\varepsilon_{2m-3} & 0 & a_{4,2m} & 0 & 0\\
-a_{51} & a_{52} & -a_{53} & \cdots & -a_{5,2m-1} & a_{5,2m} & 0 & 0\\
-a_{61} & a_{62} & -a_{63} & \cdots & -a_{6,2m-1} & a_{6,2m} & -a_{6,2m+1} & 0\\
-a_{71} & a_{72} & -a_{73} & \cdots & -a_{7,2t-1} & a_{7,2t} & 0 & 0 & 0 & 0 & 0\\
a_{81}\eta_0 & -a_{82}\eta_1 & a_{83}\eta_2 & \cdots & -a_{8,2t-2}\eta_{2t-3} & 0 & a_{8,2t} & 0 & 0 & 0 & 0\\
-a_{91} & a_{92} & -a_{93} & \cdots & -a_{9,2t-1} & -a_{9,2t+1} & 0 & 0 & 0 & 0 & 0
}
\sbmat{c_1\\c_2\\c_3\\c_4\\c_5\\ \vdots\\ \vdots\\ \vdots\\ c_{2m+2}}
=
\sbmat{4d-1\\4v-1\\1\\4t-1\\4u-1\\1\\4s-1\\4q-1\\1}
\]
where $\epsilon_j=(2m+1-j)(2m-j-1)$, $\varepsilon_j=(2m-j)(2m-j-2)$, and $\eta_j=(2t-j)(2t-j-2)$, where $\deg G_{s_i}=2m+1$, $\deg C_{s_i}=2m$, $\deg V_{s_i}=2t$, and $2m+1>2m>2t$. With the condition that any three consecutive elements in any column are equal, $a_{1j}=a_{2j}=a_{3j}$, $a_{4j}=a_{5j}=a_{6j}$, and $a_{7j}=a_{8j}=a_{9j}$. If the linear system of equations $S: Ax=b$ where $A\in M_{9\times(2m+2)}(\mathbb{Z})$ admits a solution, where $\{s_1,s_2,s_3,\dots,s_N\}\subset\mathbb{Z}^{2m+2}$ denotes the solution set of the linear system, and $s_i=(c_1,c_2,c_3,\dots,c_{2m+2})$, then there exist three algebraic curves
\[
G_{s_i}: y^2 = a_{31}c_1x^{2m+1}+a_{32}c_2x^{2m}+a_{33}c_3x^{2m-1}+\cdots+a_{3,2m+2}c_{2m+2},
\]
\[
C_{s_i}: y^2 = a_{61}c_1x^{2m}+a_{62}c_2x^{2m-1}+a_{63}c_3x^{2m-2}+\cdots+a_{6,2m+1}c_{2m+1},
\]
\[
V_{s_i}: y^2 = a_{91}c_1x^{2t}+a_{92}c_2x^{2t-1}+\cdots+a_{9,2t+1}c_{2t+1}.
\]
If $y\ge 1$, $x>1$ and $\forall(x,y)\in C_{s_i}(\mathbb{Z}^+)$ and $\forall(x,y)\in G_{s_i}(\mathbb{Z}^+)$ and $\forall(x,y)\in V_{s_i}(\mathbb{Z}^+)$ holds then
\[
\big\{C_{s_i}(\mathbb{Z}^+)\big\}^{N}_{i=1} = \emptyset, \qquad \big\{G_{s_i}(\mathbb{Z}^+)\big\}^{N}_{i=1} = \emptyset, \qquad \big\{V_{s_i}(\mathbb{Z}^+)\big\}^{N}_{i=1} = \emptyset.
\]
\end{theorem}

\begin{proof}
By THEOREM~\ref{thm4.2}, let $S: Ax=b$ be a system of linear equations over $\mathbb{Z}$ and $A\in M_{6\times(2m+2)}(\mathbb{Z})$. Let $C_{s_i},G_{s_i}\in\mathbb{Z}(x,y)$ be plane algebraic curves, nonsingular curves of genus $g\ge 0$ with integer coefficients, and $v,u,d,t\ge 0\in\mathbb{Z}$ and $M=a=D=k=1$. If the linear system $S$ is
\[
\sbmat{
-a_{11}\epsilon_0 & a_{12}\epsilon_1 & -a_{13}\epsilon_2 & \cdots & -a_{1,2m-1}\epsilon_{2m-2} & 0 & a_{2m+1} & 0\\
a_{21} & a_{22} & -a_{23} & \cdots & -a_{2,2m} & a_{2,2m+1} & 0 & \\
a_{31} & -a_{32} & a_{33} & \cdots & a_{3,2m-1} & -a_{3,2m} & a_{3,2m+1} & -a_{3,2m+2}\\
a_{41}\varepsilon_0 & -a_{42}\varepsilon_1 & a_{43}\varepsilon_2 & \cdots & -a_{4,2m-2}\varepsilon_{2m-3} & 0 & a_{4,2m} & 0 & 0\\
-a_{51} & a_{52} & -a_{53} & \cdots & -a_{5,2m-1} & a_{5,2m} & 0 & 0\\
-a_{61} & a_{62} & -a_{63} & \cdots & -a_{6,2m-1} & a_{6,2m} & -a_{6,2m+1} & 0
}
\sbmat{c_1\\c_2\\c_3\\ \vdots\\ \vdots\\ c_{2m+2}}
=
\sbmat{4d-1\\4v-1\\1\\4t-1\\4u-1\\1}
\]
where $\epsilon_j=(2m+1-j)(2m-j-1)$ and $\varepsilon_j=(2m-j)(2m-j-2)$, and there exist two algebraic curves of degree $\deg C_{s_i}=2m$ and $\deg G_{s_i}=2m+1$, and every three consecutive elements in any column are equal, $a_{1j}=a_{2j}=a_{3j}$, and $a_{4j}=a_{5j}=a_{6j}$. Let $\{s_1,s_2,s_3,\dots,s_N\}\subset\mathbb{Z}^{2m+2}$ denote the set of integer solutions of the system of linear equations $S$. Then for every solution $s_i=(c_1,c_2,\dots,c_{2m+2})$ there exist two algebraic curves $C_{s_i}$ and $G_{s_i}$ where
\[
G_{s_i}: Y^2 = a_{31}c_1X^{2m+1}+a_{32}c_2X^{2m}+a_{33}c_3X^{2m-1}+\cdots+a_{3,2m+2}c_{2m+2},
\]
\[
C_{s_i}: Y^2 = a_{61}c_1X^{2m}+a_{62}c_2X^{2m-1}+a_{63}c_3X^{2m-2}+\cdots+a_{6,2m+1}c_{2m+1}.
\]
If $y\ge 1$, $x>1$ and $\forall(x,y)\in C_{s_i}(\mathbb{Z}^+)$ and $\forall(x,y)\in G_{s_i}(\mathbb{Z}^+)$ holds then
\[
\big\{C_{s_i}(\mathbb{Z}^+)\big\}^{N}_{i=1} = \emptyset, \qquad \big\{G_{s_i}(\mathbb{Z}^+)\big\}^{N}_{i=1} = \emptyset.
\]
Now suppose that there exists another curve sharing the solutions of the linear system in THEOREM~\ref{thm4.2}. Let $V$ be a curve of degree $\deg V=2t$, where $\deg G>\deg C>\deg V$. We now consider the coefficients of the curve $V$ and adjoin to the system $S$ a new system of size $3\times 2t$. Consequently, we obtain a new system consisting of $9\times(2m+2)$ equations, and this system $S$ is
\[
\sbmat{
-a_{11}\epsilon_0 & a_{12}\epsilon_1 & -a_{13}\epsilon_2 & \cdots & -a_{1,2m-1}\epsilon_{2m-2} & 0 & a_{2m+1} & 0\\
a_{21} & a_{22} & -a_{23} & \cdots & -a_{2,2m} & a_{2,2m+1} & 0 & \\
a_{31} & -a_{32} & a_{33} & \cdots & a_{3,2m-1} & -a_{3,2m} & a_{3,2m+1} & -a_{3,2m+2}\\
a_{41}\varepsilon_0 & -a_{42}\varepsilon_1 & a_{43}\varepsilon_2 & \cdots & -a_{4,2m-2}\varepsilon_{2m-3} & 0 & a_{4,2m} & 0 & 0\\
-a_{51} & a_{52} & -a_{53} & \cdots & -a_{5,2m-1} & a_{5,2m} & 0 & 0\\
-a_{61} & a_{62} & -a_{63} & \cdots & -a_{6,2m-1} & a_{6,2m} & -a_{6,2m+1} & 0\\
-a_{71} & a_{72} & -a_{73} & \cdots & -a_{7,2t-1} & a_{7,2t} & 0 & 0 & 0 & 0 & 0\\
a_{81}\eta_0 & -a_{82}\eta_1 & a_{83}\eta_2 & \cdots & -a_{8,2t-2}\eta_{2t-3} & 0 & a_{8,2t} & 0 & 0 & 0 & 0\\
-a_{91} & a_{92} & -a_{93} & \cdots & -a_{9,2t-1} & a_{9,2t} & -a_{9,2t+1} & 0 & 0 & 0 & 0
}
\sbmat{c_1\\c_2\\c_3\\c_4\\c_5\\ \vdots\\ \vdots\\ \vdots\\ c_{2m+2}}
=
\sbmat{4d-1\\4v-1\\1\\4t-1\\4u-1\\1\\4s-1\\4q-1\\1}
\]
where $\epsilon_j=(2m+1-j)(2m-j-1)$, $\varepsilon_j=(2m-j)(2m-j-2)$, and $\eta_j=(2t-j)(2t-j-2)$, where $\deg G_{s_i}=2m+1$, $\deg C_{s_i}=2m$, and $\deg V_{s_i}=2t$, and $2m+1>2m>2t$. With the condition that any three consecutive elements in any column are equal, $a_{1j}=a_{2j}=a_{3j}$, $a_{4j}=a_{5j}=a_{6j}$, and $a_{7j}=a_{8j}=a_{9j}$. By THEOREM~\ref{thm4.2}, if the linear system of equations $S: Ax=b$ where $A\in M_{9\times(2m+2)}(\mathbb{Z})$ admits a solution, where $\{s_1,s_2,s_3,\dots,s_N\}\subset\mathbb{Z}^{2m+2}$ denotes the solution of the linear system, and $s_i=(c_1,c_2,c_3,\dots,c_{2m+2})$, then there exist three algebraic curves
\[
G_{s_i}: y^2 = a_{31}c_1x^{2m+1}+a_{32}c_2x^{2m}+a_{33}c_3x^{2m-1}+\cdots+a_{3,2m+2}c_{2m+2},
\]
\[
C_{s_i}: y^2 = a_{61}c_1x^{2m}+a_{62}c_2x^{2m-1}+a_{63}c_3x^{2m-2}+\cdots+a_{6,2m+1}c_{2m+1},
\]
\[
V_{s_i}: y^2 = a_{91}c_1x^{2t}+a_{92}c_2x^{2t-1}+\cdots+a_{9,2t+1}c_{2t+1}.
\]
If $y\ge 1$, $x>1$ and $\forall(x,y)\in C_{s_i}(\mathbb{Z}^+)$ and $\forall(x,y)\in G_{s_i}(\mathbb{Z}^+)$ and $\forall(x,y)\in V_{s_i}(\mathbb{Z}^+)$ holds then
\[
\big\{C_{s_i}(\mathbb{Z}^+)\big\}^{N}_{i=1} = \emptyset, \qquad \big\{G_{s_i}(\mathbb{Z}^+)\big\}^{N}_{i=1} = \emptyset, \qquad \big\{V_{s_i}(\mathbb{Z}^+)\big\}^{N}_{i=1} = \emptyset. \qedhere
\]
\end{proof}

\begin{theorem}\label{thm4.4}
Let $S: Ax=b$ be a system of linear equations over $\mathbb{Z}$ and $A\in M_{9\times(2m+1)}(\mathbb{Z})$. Let $C_{s_i},G_{s_i},V_{s_i}\in\mathbb{Z}(x,y)$ be plane algebraic curves, nonsingular curves of genus $g\ge 0$. If the linear system $S$ is the same coefficient matrix as in THEOREM~\ref{thm4.3}, with right-hand si

\[
\sbmat{
-a_{11}\epsilon_0 & a_{12}\epsilon_1 & -a_{13}\epsilon_2 & \cdots & -a_{1,2m-1}\epsilon_{2m-2} & 0 & a_{2m+1} & 0\\
a_{21} & a_{22} & -a_{23} & \cdots & -a_{2,2m} & a_{2,2m+1} & 0 & \\
a_{31} & -a_{32} & a_{33} & \cdots & a_{3,2m-1} & -a_{3,2m} & a_{3,2m+1} & -a_{3,2m+2}\\
a_{41}\varepsilon_0 & -a_{42}\varepsilon_1 & a_{43}\varepsilon_2 & \cdots & -a_{4,2m-2}\varepsilon_{2m-3} & 0 & a_{4,2m} & 0 & 0\\
-a_{51} & a_{52} & -a_{53} & \cdots & -a_{5,2m-1} & a_{5,2m} & 0 & 0\\
-a_{61} & a_{62} & -a_{63} & \cdots & -a_{6,2m-1} & a_{6,2m} & -a_{6,2m+1} & 0\\
-a_{71} & a_{72} & -a_{73} & \cdots & -a_{7,2t-1} & a_{7,2t} & 0 & 0 & 0 & 0 & 0\\
a_{81}\eta_0 & -a_{82}\eta_1 & a_{83}\eta_2 & \cdots & -a_{8,2t-2}\eta_{2t-3} & 0 & a_{8,2t} & 0 & 0 & 0 & 0\\
-a_{91} & a_{92} & -a_{93} & \cdots & -a_{9,2t-1} & a_{9,2t} & -a_{9,2t+1} & 0 & 0 & 0 & 0
}
\sbmat{c_1\\c_2\\c_3\\c_4\\c_5\\ \vdots\\ \vdots\\ \vdots\\ c_{2m+2}}
=
\sbmat{-1\\-1\\1\\-1\\-1\\1\\-1\\-1\\1}
\]
where $\epsilon_j=(2m+1-j)(2m-j-1)$, $\varepsilon_j=(2m-j)(2m-j-2)$, and $\eta_j=(2t-j)(2t-j-2)$, where $\deg G_{s_i}=2m+1$, $\deg C_{s_i}=2m$, $\deg V_{s_i}=2t$, and $2m+1>2m>2t$, with the condition that any three consecutive elements in any column are equal, $a_{1j}=a_{2j}=a_{3j}$, $a_{4j}=a_{5j}=a_{6j}$, and $a_{7j}=a_{8j}=a_{9j}$. If the linear system $S: Ax=b$ admits a solution, where $\{s_1,s_2,s_3,\dots,s_N\}\subset\mathbb{Z}^{2m+2}$ denotes the solution set, and $s_i=(c_1,c_2,c_3,\dots,c_{2m+2})$, then there exist three algebraic curves
\[
G_{s_i}: y^2 = a_{31}c_1x^{2m+1}+a_{32}c_2x^{2m}+a_{33}c_3x^{2m-1}+\cdots+a_{3,2m+2}c_{2m+2},
\]
\[
C_{s_i}: y^2 = a_{61}c_1x^{2m}+a_{62}c_2x^{2m-1}+a_{63}c_3x^{2m-2}+\cdots+a_{6,2m+1}c_{2m+1},
\]
\[
V_{s_i}: y^2 = a_{91}c_1x^{2t}+a_{92}c_2x^{2t-1}+\cdots+a_{9,2t+1}c_{2t+1}.
\]
If $y\ge 1$, $x>1$ and $\forall(x,y)\in C_{s_i}(\mathbb{Z}^+)$ and $\forall(x,y)\in G_{s_i}(\mathbb{Z}^+)$ and $\forall(x,y)\in V_{s_i}(\mathbb{Z}^+)$ holds then
\[
\big\{C_{s_i}(\mathbb{Z}^+)\big\}^{N}_{i=1} = \emptyset, \qquad \big\{G_{s_i}(\mathbb{Z}^+)\big\}^{N}_{i=1} = \emptyset, \qquad \big\{V_{s_i}(\mathbb{Z}^+)\big\}^{N}_{i=1} = \emptyset.
\]
\end{theorem}

\begin{proof}
For THEOREM~\ref{thm4.3}, take $v=u=d=t=s=q=0$ and $M=a=k=D=W=1$. The proof is now complete.
\end{proof}

\begin{theorem}\label{thm4.5}
Let $S: Ax=b$ be a system of linear equations over $\mathbb{Z}$ and $A\in M_{9\times(2m+2)}(\mathbb{Z})$. Let $C_{s_i},G_{s_i},V_{s_i}\in\mathbb{Z}(x,y)$ be plane algebraic curves, where $v,u,d,t,s,q\ge 0\in\mathbb{Z}$. If the linear system $S$ is
\[
\sbmat{
-\epsilon_0 & \epsilon_1 & -\epsilon_2 & \cdots & -\epsilon_{2m-2} & 0 & 1 & 0\\
-1 & 1 & -1 & \cdots & -1 & 1 & 0 & \\
1 & -1 & 1 & \cdots & 1 & -1 & 1 & -1\\
\varepsilon_0 & -\varepsilon_1 & \varepsilon_2 & \cdots & -\varepsilon_{2m-3} & 0 & 1 & 0 & 0\\
-1 & 1 & -1 & \cdots & -1 & 1 & 0 & 0\\
-1 & 1 & -1 & \cdots & -1 & 1 & -1 & 0\\
-1 & 1 & -1 & \cdots & -1 & 1 & 0 & 0 & 0 & 0 & 0\\
\eta_0 & -\eta_1 & \eta_2 & \cdots & -\eta_{2t-3} & 0 & 1 & 0 & 0 & 0 & 0\\
-1 & 1 & -1 & \cdots & -1 & 1 & -1 & 0 & 0 & 0 & 0
}
\sbmat{c_1\\c_2\\c_3\\c_4\\c_5\\ \vdots\\ \vdots\\ \vdots\\ c_{2m+2}}
=
\sbmat{4d-1\\4v-1\\1\\4t-1\\4u-1\\1\\4s-1\\4q-1\\1}
\]
where $\epsilon_j=(2m-j)(2m-j-2)$, $\varepsilon_j=(2m+1-j)(2m+1-j-2)$, and $\eta_j=(2t-j)(2t-j-2)$, where $\deg G_{s_i}=2m+1$, $\deg C_{s_i}=2m$, $\deg V_{s_i}=2t$, and $2m+1>2m>2t$. If the linear system $S: Ax=b$ where $A\in M_{9\times(2m+2)}(\mathbb{Z})$ admits a solution, where $\{s_1,s_2,s_3,\dots,s_N\}\subset\mathbb{Z}^{2m+2}$ denotes the solution of the linear system, and $s_i=(c_1,c_2,c_3,\dots,c_N)$, then there exist three algebraic curves
\[
C_{s_i}: y^2 = c_1x^{2m}+c_2x^{2m-1}+c_3x^{2m-2}+\cdots+c_{2m+1},
\]
\[
G_{s_i}: y^2 = c_1x^{2m+1}+c_2x^{2m}+q_{23}c_3x^{2m-1}+\cdots+c_{2m+2},
\]
\[
V_{s_i}: y^2 = c_1x^{2t}+c_2x^{2t-1}+\cdots+c_{2t+1}.
\]
If $y\ge 1$, $x>1$ and $\forall(x,y)\in C_{s_i}(\mathbb{Z}^+)$ and $\forall(x,y)\in G_{s_i}(\mathbb{Z}^+)$ and $\forall(x,y)\in V_{s_i}(\mathbb{Z}^+)$ holds then
\[
C_{s_i}(\mathbb{Z}^+) = G_{s_i}(\mathbb{Z}^+) = V_{s_i}(\mathbb{Z}^+) = \emptyset
\]
for every $s_i$.
\end{theorem}

\begin{proof}
By THEOREM~\ref{thm4.4}, any three consecutive entries in a given column are equal. We observe that $a_{1j}=a_{2j}=a_{3j}$, $a_{4j}=a_{5j}=a_{6j}$, and $a_{7j}=a_{8j}=a_{9j}$. Hence, for THEOREM~\ref{thm4.4}, let all entries in the column be equal to 1, $a_{1j}=a_{2j}=a_{3j}=a_{4j}=a_{5j}=a_{6j}=a_{7j}=a_{8j}=a_{9j}=1$. This completes the proof.
\end{proof}

\begin{theorem}\label{thm4.6}
Let $S: Ax=b$ be a system of linear equations with $A\in M_{3k\times m}(\mathbb{Z})$, such that the system $S$ satisfies the following condition: for every column, any three consecutive entries are equal, $a_{3i-1,j}=a_{3i-2,j}=a_{3i,j}$. If $S=\{s_1,s_2,s_3,\dots,s_N\}\subset\mathbb{Z}^{m+1}$ is a solution of the system $S$, then there exist nonsingular curves of genus $g\ge 0$ of degree $n\le m$
\[
C_{js_i}: Y^2 = a_{3j1}c_1X^n + a_{3j2}c_2X^{n-1}+\cdots+a_{n+1}c_{n+1}.
\]
If $y\ge 1$, $x>1$ and $\forall(x,y)\in C_{js_i}(\mathbb{Z}^+)$ holds then
\[
\Big[C_{js_i}(\mathbb{Z}^+)\Big]^{N}_{\substack{i=1\\1\le j\le k}} = \emptyset.
\]
If $N\Rightarrow\infty$ then
\[
\Big[C_{js_i}(\mathbb{Z}^+)\Big]^{\infty}_{\substack{i=1\\1\le j\le k}} = \emptyset.
\]
\end{theorem}

\begin{proof}
By THEOREM~\ref{thm4.1}, let $S: Ax=b$ be a system of linear equations over $\mathbb{Z}$, and $A\in M_{3\times(n+1)}(\mathbb{Z})$. And $C_{s_i}\in\mathbb{Z}(x,y)$ be plane algebraic curves, and $k\ge 1$ is an odd positive integer and every prime divisor of $K$ satisfies $p\equiv 1\pmod 4$, and $U,D\ge 0\in\mathbb{Z}$. If the linear system $S$ is
\[
\sbmat{
(-1)^{n-2}a_{11}\epsilon_0 & (-1)^{n-3}a_{12}\epsilon_1 & \cdots & (-1)^{1}a_{1,n-2}\epsilon_{n-3} & 0 & a_{1n} & 0 \\
(-1)^{n+1}a_{21} & (-1)^{n}a_{22} & \cdots & (-1)^{3}a_{2,n-1} & (-1)^{2}a_{2n} & 0 & \\
(-1)^{n-1}a_{31} & (-1)^{n-2}a_{32} & \cdots & (-1)^{3}a_{3,n-1} & (-1)^{2}a_{3n} & (-1)^{1}a_{3,n+1} &
}
\sbmat{c_1\\c_2\\c_3\\ \vdots \\ \vdots \\ c_{n+1}}
=
\sbmat{4U-1\\4D-1\\K^2}
\]
where $\epsilon_j=(n-j)(n-j-2)$, and any three elements in any column are equal, $a_{1j}=a_{2j}=a_{3j}$, and $n$ denotes the degree of the curve $C$, $\deg C=n$. If there is a solution to the system of linear equations $S$, and the solution is $S=\{s_1,s_2,\dots,s_N\}\subset\mathbb{Z}^{n+1}$ where $s_i=(c_1,c_2,\dots,c_{n+1})$, then the following curve exists:
\[
C_{s_i}: y^2 = a_{31}c_1x^n + a_{32}c_2x^{n-1} + a_{33}c_3x^{n-2}+\cdots+a_{3,n+1}c_{n+1}.
\]
If $y\ge 1$, $x>1$ and $\forall(x,y)\in C_{s_i}(\mathbb{Z}^+)$ holds then
\[
\big\{C_{s_i}(\mathbb{Z}^+)\big\}^{N}_{i=1} = \emptyset.
\]
Now let there be $K$ algebraic curves. Consequently, one may construct a system of linear equations using the coefficients of the $K$ curves. As a result we obtain a system $S: Ax=b$ of dimension $3k\times m$ with $A\in M_{3k\times m}(\mathbb{Z})$. This implies that all curves have degree $\le m$, and let there exist common solutions to the system, $S=\{s_1,s_2,\dots,s_N\}\subset\mathbb{Z}^m$. It then follows that if $s_i=(c_1,c_2,\dots,c_n)$ is a solution of the system, then there exist curves
\[
C_{js_i}: Y^2 = a_{3j1}c_1X^n + a_{3j2}c_2X^{n-1}+\cdots+a_{3jn}c_n.
\]
If $y\ge 1$, $x>1$ and $\forall(x,y)\in C_{js_i}(\mathbb{Z}^+)$ holds then
\[
\big\{C_{js_i}(\mathbb{Z}^+)\big\}^{k}_{j=1} = \emptyset.
\]
But there exist $N$ solutions. Hence
\[
\Big\{C_{js_i}(\mathbb{Z}^+)\Big\}^{N}_{\substack{i=1\\1\le j\le k}} = \emptyset.
\]
Now suppose there exist infinitely many solutions such that $N\Rightarrow\infty$, so
\[
\Big\{C_{js_i}(\mathbb{Z}^+)\Big\}^{\infty}_{\substack{i=1\\1\le j\le k}} = \emptyset. \qedhere
\]
\end{proof}

\section{The Relationship with Diophantine Equations}

In this section, we study the relationship among linear equations, quadratic and cubic forms, and solutions of algebraic curves. This provides a deeper and more general extension of the preceding ideas. We prove that if a common solution exists for a system consisting of linear equations and quadratic forms in $n$ terms, then an algebraic curve with no positive integer solutions can be constructed; this result is stated in THEOREM 5.1. The proof of THEOREM 5.1 is based on THEOREM 3.2. We further extend the relationship to cubic forms in $n$ terms, as established in THEOREM 5.2. In general, we formulate linear, quadratic, and cubic equations in the language of linear algebra through a matrix-based construction. We also describe several alternative forms of this matrix construction in THEOREM 5.3.

\begin{theorem}\label{thm5.1}
Let $C: MY^2 = a_1 X^{x_1} + a_2 X^{x_2} + \cdots + a_m X^{x_m} + a_{m+1}$ be an algebraic curve, nonsingular curve of genus $g\ge 0$, with coefficients $a_0,a_1,a_2,\dots,a_{m+1}\in\mathbb{Z}$, where $M,K\ge 1$ are odd positive integers such that all the prime divisors of $M,K$ satisfy $p\equiv 1\pmod 4$, and $\deg C=\max(x_1,x_2,\dots,x_m)$. Then, if
\[
\sum_{j=1}^{m}(-1)^{x_j+1}a_j - a_{m+1} = Mk^2, \qquad \sum_{j=1}^{m} a_j(-1)^{x_j+1} = 4k-1,
\]
\[
\sum_{j=1}^{m}(-1)^{x_j-2}a_jx_j = -2s, \qquad \sum_{j=1}^{m}(-1)^{x_j-2}a_jx_j^2 = 4c-1, \qquad s\ne 0,\ c,k\in\mathbb{Z},
\]
and if $y\ge 1$, $x>1$, $\forall(x,y)\in C(\mathbb{Z}^+)$ holds, then
\[
C(\mathbb{Z}^+) = \emptyset.
\]
\end{theorem}

\begin{proof}
By THEOREM~\ref{thm3.2}, let $C: My^2 = a_0x^n+a_1x^{n-1}+\cdots+a_{n-1}x+a_n$ be an algebraic curve, nonsingular of genus $g\ge 0$, with coefficients $a_0,a_1,a_2,\dots,a_n\in\mathbb{Z}$, where $M,k$ are odd constants, such that all the prime divisors of $M,K$ satisfy $p\equiv 1\pmod 4$, and $C(\mathbb{Z}^+)$ are the positive integer solutions for curve $C$. Then, if
\[
\sum_{j=0}^{n}(-1)^{n-j+1}a_j = Mk^2, \qquad \sum_{j=0}^{n-1} a_j(-1)^{n-j+1} = 4k-1,
\]
\[
\sum_{j=0}^{n-3}(-1)^{n-j-2}a_j(n-j)(n-j-2)+a_{n-1} = 4U-1, \qquad U,K\ge0\in\mathbb{Z},
\]
if $y\ge 1$, $x>1$, $\forall(x,y)\in C(\mathbb{Z}^+)$ holds, then $C(\mathbb{Z}^+)=\emptyset$.

Consider the equation
\[
\sum_{j=0}^{n-3}(-1)^{n-j-2}a_j(n-j)(n-j-2)+a_{n-1} = 4U-1.
\]
We observe that each term $(n-j)$ corresponds to the weight of a monomial of the curve $C$, where $n$ denotes the weight of the leading term and $(n-j)$ denotes the weight of the $(n-j)$-th term.

Now, suppose that $a_0=a_1=a_2=a_3=\cdots=a_{k-1}=0$. This implies the relations
\[
a_0(n-0)(n-0-2)=a_1(n-1)(n-1-2)=a_2(n-2)(n-2-2)=\cdots=a_{k-1}(n-k+1)(n-k+1-2)=0.
\]
Set
\[
a_k(n-k) = a_1x_1.
\]
In the same manner, suppose that $a_{k+1}=a_{k+2}=a_{k+3}=a_{k+4}=\cdots=a_{k+r-1}=0$. Define
\[
a_{k+1}(n-k-1)(n-k-1-2)=a_{k+2}(n-k-2)(n-k-2-2)=\cdots=a_{k+r-1}(n-k-r+1)(n-k-r+1-2)=0.
\]
Define
\[
a_{k+r}(n-k-r) = a_2x_2.
\]
The idea is to assign the value zero to certain coefficients $a_j$ until we reach $a_{k+r+t}(n-k-r-t)=a_3x_3$, and similarly $a_{k+r+t+u}(n-k-r-t-u)=a_4x_4$, continuing in this fashion until we obtain all the values $x_1,x_2,x_3,x_4,\dots,x_m$. As a consequence, we arrive at the following equations. Observe that
\[
\sum_{j=1}^{m}(-1)^{x_j+1}a_j - a_{m+1} = Mk^2, \qquad \sum_{j=1}^{m}a_j(-1)^{x_j+1} = 4k-1,
\]
\[
\sum_{j=1}^{m}(-1)^{x_j-2}a_jx_j(x_j-2) = 4U-1.
\]
Note that
\[
\sum_{j=1}^{m}(-1)^{x_j-2}a_jx_j(x_j-2) = \sum_{j=1}^{m}(-1)^{x_j-2}a_jx_j^2 - 2\sum_{j=1}^{m}(-1)^{x_j-2}a_jx_j = 4U-1 = 4(t\pm s)-1 \quad \text{because } U\in\mathbb{Z}.
\]
Then let
\[
-2\sum_{j=1}^{m}(-1)^{x_j-2}a_jx_j = 4s.
\]
Then
\[
\sum_{j=1}^{m}(-1)^{x_j-2}a_jx_j = -2s, \qquad \sum_{j=1}^{m}(-1)^{x_j-2}a_jx_j^2 = 4t-1.
\]
Therefore, by THEOREM~\ref{thm3.2}, if the above system admits an integral solution, then
\[
C: My^2 = a_1X^{x_1}+a_2X^{x_2}+\cdots+a_mX^{x_m}+a_{m+1}.
\]
If $x>1$, $y\ge 1$ and $\forall(x,y)\in C(\mathbb{Z}^+)$ holds then
\[
C(\mathbb{Z}^+) = \emptyset. \qedhere
\]
\end{proof}

\begin{theorem}\label{thm5.2}
Let $L: Ax=b$ and $Q: Px=d$ and $V: Bx=m$ be systems of linear and quadratic equations, where $A\in M_{2\times m}(\mathbb{Z})$, $P\in M_m(\mathbb{Z})$ and $B\in M_m(\mathbb{Z})$, and each system shares the common element $A=(a_{ij})=P=(a_{ij})=B=(a_{ij})$. Suppose that $S_L,S_Q,S_V$ are solutions of the respective systems with $S_L,S_Q,S_V\subset\mathbb{Z}^m$, and that each system satisfies the following:
\[
L:\ \sbmat{a_{11} & a_{12} & a_{13} & \cdots & a_{1m}\\ a_{21} & a_{22} & a_{23} & \cdots & a_{2m}}
\sbmat{(-1)^{x_1+1}\\(-1)^{x_2+1}\\(-1)^{x_3+1}\\ \vdots\\ \vdots\\ (-1)^{x_m+1}}
=
\sbmat{Mk^2\\4u-1}
\]
In each column, any two consecutive entries are equal, $a_{1j}=a_{2j}$. Moreover, $M,K\ge 1$ are odd positive integers, and every prime divisor of $M,k$ is of the form $4n+1$. Denote by $Q: Px=d$ the system of quadratic equations
\[
D_1 =
\sbmat{
(-1)^{x_1-2}a_{11} & 0 & 0 & 0 & \cdots & 0\\
0 & (-1)^{x_2-2}a_{22} & 0 & 0 & \cdots & 0\\
0 & 0 & (-1)^{x_3-2}a_{33} & 0 & \cdots & 0\\
0 & 0 & 0 & (-1)^{x_4-2}a_{44} & \cdots & 0\\
\vdots & & & & \ddots & \\
0 & 0 & 0 & 0 & \cdots & (-1)^{x_m-2}a_{mm}
}
\]
where $D_1\in\mathbb{Z}^{m\times m}$ and $t\in\mathbb{Z}$, and
\[
Q: [X^T D_1 X] = [4t-1].
\]
And $V: Bx=m$ is the system of linear equations, equal to
\[
V:\ \sbmat{a_{11} & a_{12} & a_{13} & \cdots & a_{1m}}
\sbmat{(-1)^{x_1-2}x_1\\(-1)^{x_1-2}x_2\\(-1)^{x_1-2}x_3\\ \vdots\\ \vdots\\ (-1)^{x_1-2}x_m}
= [-2s].
\]
If there exists a common solution to these systems and $S_L\cap S_Q\cap S_V = \{s_1,s_2,s_3,\dots,s_N\}\subset\mathbb{Z}^m$, and $s_i=(x_1,x_2,x_3,\dots,x_m)$, then there exists a curve, nonsingular curve of genus $g\ge 0$, where
\[
C_{s_i}: My^2 = a_1X^{x_1}+a_2X^{x_2}+\cdots+a_mX^{x_m}.
\]
If $y\ge 1$, $x>1$, $\forall(x,y)\in C(\mathbb{Z}^+)$ then
\[
C_{s_i}(\mathbb{Z}^+) = \emptyset
\]
for every $s_i$.
\end{theorem}

\begin{proof}
By THEOREM~\ref{thm5.1}, let $C: My^2 = a_1X^{x_1}+a_2X^{x_2}+\cdots+a_mX^{x_m}+a_{m+1}$ be an algebraic curve, nonsingular curve of genus $g\ge 0$, with coefficients $a_0,a_1,a_2,\dots,a_{m+1}\in\mathbb{Z}$, where $M,K$ are odd positive integers such that all the prime divisors of $M,K$ satisfy $p\equiv 1\pmod 4$. Then if
\[
\sum_{j=1}^{m}(-1)^{x_j+1}a_j = Mk^2, \qquad \sum_{j=1}^{m}a_j(-1)^{x_j+1} = 4k-1,
\]
\[
\sum_{j=1}^{m}(-1)^{x_j-2}a_jx_j = -2m, \qquad \sum_{j=1}^{m}(-1)^{x_j-2}a_jx_j^2 = 4c-1,
\]
and if $y\ge 1$, $x>1$, $\forall(x,y)\in C(\mathbb{Z}^+)$ holds then $C(\mathbb{Z}^+)=\emptyset$.

We now express this system of equations in matrix form. Observe that
\[
\sum_{j=1}^{m}(-1)^{x_j+1}a_j = Mk^2, \qquad \sum_{j=1}^{m}a_j(-1)^{x_j+1} = 4k-1.
\]
Where $a_{m+1}=0$, then
\[
P:\ \sbmat{a_{11} & a_{12} & a_{13} & \cdots & a_{1m}\\ a_{21} & a_{22} & a_{23} & \cdots & a_{2m}}
\sbmat{(-1)^{x_1+1}\\(-1)^{x_2+1}\\(-1)^{x_3+1}\\ \vdots\\ \vdots\\ (-1)^{x_m+1}}
=
\sbmat{Mk^2\\4u-1}.
\]
In each column, any two consecutive entries are equal, $a_{1j}=a_{2j}$. And
\[
\sum_{j=1}^{m}(-1)^{x_j-2}a_jx_j^2 = 4t-1,
\]
\[
D_1 =
\sbmat{
(-1)^{x_1-2}a_{11} & 0 & 0 & 0 & \cdots & 0\\
0 & (-1)^{x_2-2}a_{22} & 0 & 0 & \cdots & 0\\
0 & 0 & (-1)^{x_3-2}a_{33} & 0 & \cdots & 0\\
0 & 0 & 0 & (-1)^{x_4-2}a_{44} & \cdots & 0\\
\vdots & & & & \ddots & \\
0 & 0 & 0 & 0 & \cdots & (-1)^{x_m-2}a_{mm}
}
\]
where
\[
Q: [X^T D_1 X] = [4t-1],
\]
and
\[
\sum_{j=1}^{m}(-1)^{x_j-2}a_jx_j = -2s.
\]
\[
V:\ \sbmat{a_{11} & a_{12} & a_{13} & \cdots & a_{1m}}
\sbmat{(-1)^{x_1-2}x_1\\(-1)^{x_1-2}x_2\\(-1)^{x_1-2}x_3\\ \vdots\\ \vdots\\ (-1)^{x_1-2}x_m}
= [-2s].
\]
Then $L: Ax=b$ and $Q: Px=d$ and $V: Bx=m$, where $A\in M_{2\times m}(\mathbb{Z})$, $P\in M_m(\mathbb{Z})$ and $B\in M_m(\mathbb{Z})$, and each system shares the common element $A=(a_{ij})=P=(a_{ij})=B=(a_{ij})$. Suppose that $S_L,S_Q,S_V$ are solutions of the respective systems with $S_L,S_Q,S_V\subset\mathbb{Z}^m$. Now, suppose there is a common solution to $S_L,S_Q,S_V$, where $S_Q\cap S_L\cap S_V = s_1=(x_1,x_2,x_3,\dots,x_m)$. Then, according to THEOREM~\ref{thm5.1}, there exists a curve
\[
C_{s_1}: My^2 = a_1X^{x_1}+a_2X^{x_2}+\cdots+a_mX^{x_m}.
\]
If $x>1$, $y\ge 1$ and $\forall(x,y)\in C(\mathbb{Z}^+)$ holds then
\[
C_{s_1}(\mathbb{Z}^+) = \emptyset.
\]
If we assume that there exist $N$ common solutions $S_Q\cap S_L\cap S_V = \{s_1,s_2,s_3,\dots,s_N\}\in\mathbb{Z}^m$, then there exist $N$ algebraic curves such that
\[
C_{s_i}: My^2 = a_1X^{x_1}+a_2X^{x_2}+\cdots+a_mX^{x_m}.
\]
If $x>1$, $y\ge 1$ and $\forall(x,y)\in C(\mathbb{Z}^+)$ then
\[
C_{s_i}(\mathbb{Z}^+) = \emptyset
\]
for every $s_i$.
\end{proof}

\begin{theorem}\label{thm5.3}
Let $L: Ax=b$ and $Q: Px=d$ and $V: Bx=m$ be systems of linear and quadratic, and cubic equations, where $A\in M_{2\times m}(\mathbb{Z})$, $P\in M_{1\times m}(\mathbb{Z})$ and $B\in M_{1\times m}(\mathbb{Z})$, where all coefficients in the system are equal. Then $A=(a_{ij})=P=(a_{ij})=B=(a_{ij})$, and $S_Q,S_L,S_V\in\mathbb{Z}^m$ are solutions of the system. Let each system satisfy the following:
\[
L:\ \sbmat{a_{11} & a_{12} & a_{13} & \cdots & a_{1m}\\ a_{21} & a_{22} & a_{23} & \cdots & a_{2m}}
\sbmat{(-1)^{x_1+1}x_1\\(-1)^{x_2+1}x_2\\ \vdots\\ \vdots\\ (-1)^{x_m+1}x_m}
=
\sbmat{Mk^2\\4t-1}.
\]
In each column, any two consecutive entries are equal, $a_{1j}=a_{2j}$. Moreover, $M,k\ge 1$ are odd integers, and every prime divisor of $M,k$ is of the form $p=4n+1$, $t\in\mathbb{Z}$, and $Q: Px=d$ is $T=\{T_1,T_2,T_3,\dots,T_m\}\in\mathbb{Z}^{m\times m\times m}$ and $T_j$ is given by
\[
T_j \in \mathbb{Z}^{m\times m}, \qquad (T_j)_{k\ell} = \begin{cases} (-1)^{x_j-2}a_{jj} & \text{if } k=\ell=j,\\ 0 & \text{otherwise,} \end{cases} \qquad j=1,2,\dots,m,
\]
(i.e.\ $T_j$ is the all-zero matrix except for the single diagonal entry $(-1)^{x_j-2}a_{jj}$ in position $(j,j)$), where
\[
C:\ \big[X^T T_1 X,\ X^T T_2 X,\ X^T T_3 X,\ \dots,\ X^T T_m X\big]
\sbmat{x_1\\x_2\\x_3\\ \vdots\\ \vdots\\ x_m}
= [4k-1],
\]
where $k\in\mathbb{Z}$, and $V$: is given by the diagonal matrix
\[
D_1 = \operatorname{diag}\Big((-1)^{x_1}a_{11},\, (-1)^{x_2}a_{22},\, (-1)^{x_3}a_{33},\, \dots,\, (-1)^{x_m}a_{mm}\Big) \in \mathbb{Z}^{m\times m},
\]
and
\[
Q: [X^T D_1 X] = [4t-1].
\]
Now, if there exists a common solution among $S_L\cap S_Q\cap S_C = S = \{s_1,s_2,s_3,\dots,s_N\}\in\mathbb{Z}^m$, and if $s_i=(x_1,x_2,\dots,x_m)\in S$, then there exists a curve, nonsingular curve of genus $g$, where
\[
C_{s_i}: My^2 = a_1x_1X^{x_1} + a_2x_2X^{x_2} + \cdots + a_mx_mX^{x_m}.
\]
If $y\ge 1$, $x>1$, $\forall(x,y)\in C_{s_i}(\mathbb{Z}^+)$ holds then
\[
C_{s_i}(\mathbb{Z}^+) = \emptyset
\]
for every $s_i$.
\end{theorem}

\begin{proof}
According to THEOREM~\ref{thm5.2}, let $C: My^2 = a_1X^{x_1}+a_2X^{x_2}+\cdots+a_mX^{x_m}+a_{m+1}$ be an algebraic curve with coefficients $a_0,a_1,a_2,\dots,a_m\in\mathbb{Z}$, be such that all the prime divisors of $M,K$ satisfy $p\equiv 1\pmod 4$. Then if
\[
\sum_{j=1}^{m}(-1)^{x_j+1}a_j = Mk^2, \qquad \sum_{j=1}^{m}a_j(-1)^{x_j+1} = 4k-1,
\]
\[
\sum_{j=1}^{m}(-1)^{x_j-2}a_jx_j = -2m, \qquad \sum_{j=1}^{m}(-1)^{x_j-2}a_jx_j^2 = 4c-1,
\]
and $y\ge 1$, $x>1$, $\forall(x,y)\in C(\mathbb{Z}^+)$ holds then $C(\mathbb{Z}^+)=\emptyset$.

Now, let $a_j = b_jx_j$ represent $j=1,2,3,\dots,m$, and $b_j,x_j\in\mathbb{Z}$; consequently, we obtain the following new equations:
\[
\sum_{j=1}^{m}(-1)^{x_j+1}b_jx_j = Mk^2, \qquad \sum_{j=1}^{m}(-1)^{x_j+1}b_jx_j = 4k-1,
\]
\[
\sum_{j=1}^{m}(-1)^{x_j-2}b_jx_j^2 = -2m, \qquad \sum_{j=1}^{m}(-1)^{x_j-2}b_jx_j^3 = 4c-1,
\]
where
\[
C: My^2 = a_1x_1X^{x_1}+a_2x_2X^{x_2}+\cdots+a_mx_mX^{x_m}.
\]
Next, we arrange all these equations in matrix form using linear algebra. Thus
\[
L:\ \sbmat{a_{11} & a_{12} & a_{13} & \cdots & a_{1m}\\ a_{21} & a_{22} & a_{23} & \cdots & a_{2m}}
\sbmat{(-1)^{x_1+1}x_1\\(-1)^{x_2+1}x_2\\ \vdots\\ \vdots\\ (-1)^{x_m+1}x_m}
=
\sbmat{mk^2\\4t-1}.
\]
In each column, any two consecutive entries are equal, $a_{1j}=a_{2j}$. And, with $D_1=\operatorname{diag}\big((-1)^{x_1-2}a_{11},(-1)^{x_2-2}a_{22},\dots,(-1)^{x_m-2}a_{mm}\big)\in\mathbb{Z}^{m\times m}$ and $t\in\mathbb{Z}$, where
\[
Q: [X^T D_1 X] = [4t-1],
\]
and, with $T_j\in\mathbb{Z}^{m\times m}$ the matrix whose only nonzero entry is $(-1)^{x_j-2}a_{jj}$ in position $(j,j)$ (as before), $[T_1,T_2,T_3,\dots,T_m]\in\mathbb{Z}^{m\times m\times m}$, where
\[
C:\ \big[X^TT_1X\ \ X^TT_2X\ \ X^TT_3X\ \cdots\ X^TT_mX\big]
\sbmat{x_1\\x_2\\x_3\\ \vdots\\ \vdots\\ x_m}
= [4k-1].
\]

Let $L: Ax=b$ and $Q: Px=d$ and $V: Bx=m$ be systems of linear and quadratic, and cubic equations, where $A\in M_{2\times m}(\mathbb{Z})$, $P\in M_{1\times m}(\mathbb{Z})$ and $B\in M_{1\times m}(\mathbb{Z})$, where all
coefficients in the system are equal. Then $A=(a_{ij})=P=(a_{ij})=B=(a_{ij})$, and $S_Q,S_L,S_V\in\mathbb{Z}^m$ are solutions of the system. Let $S_Q\cap S_L\cap S_V = s_1=(x_1,x_2,x_3,\dots,x_m)$ be a common solution; then, according to THEOREM~\ref{thm5.2}, there exists a curve
\[
C_{s_1}: My^2 = a_{11}x_1X^{x_1}+a_{12}x_2X^{x_2}+\cdots+a_{1m}x_mX^{x_m}.
\]
If $y\ge 1$, $x>1$, $\forall(x,y)\in C(\mathbb{Z}^+)$ then
\[
C_{s_1}(\mathbb{Z}^+) = \emptyset.
\]
Similarly, suppose there exists another solution -- let it be $S_L\cap S_Q\cap S_C = \{s_1,s_2,s_3,\dots,s_N\}\subset\mathbb{Z}^m$ where $s_i=(x_1,x_2,x_3,\dots,x_m)$; then, according to THEOREM~\ref{thm5.2}, there exists a curve
\[
C_{s_i}: My^2 = a_{11}x_1X^{x_1}+a_{12}x_2X^{x_2}+\cdots+a_{1m}x_mX^{x_m}.
\]
If $y\ge 1$, $x>1$, $\forall(x,y)\in C(\mathbb{Z}^+)$ then
\[
C_{s_i}(\mathbb{Z}^+) = \emptyset
\]
for every $s_i$.
\end{proof}

\begin{theorem}\label{thm5.4}
Let $L: Ax=b$ and $Q: Px=d$ and $V: Bx=m$ be systems of linear and quadratic, and cubic equations, where $A\in M_{2\times m}(\mathbb{Z})$, $P\in M_{1\times m}(\mathbb{Z})$ and $B\in M_{1\times m}(\mathbb{Z})$, where all coefficients in the system are equal. And $S_Q,S_L,S_V\in\mathbb{Z}^m$ are solutions of the system. Let each system satisfy the following:
\[
L:\ \sbmat{1 & 1 & 1 & \cdots & 1\\ 1 & 1 & 1 & \cdots & 1}
\sbmat{(-1)^{x_1+1}x_1\\(-1)^{x_2+1}x_2\\ \vdots\\ \vdots\\ (-1)^{x_m+1}x_m}
=
\sbmat{mk^2\\4t-1}.
\]
Moreover, $M,k$ are odd integers, and every prime divisor of $M,k$ is of the form $p=4n+1$, $t\in\mathbb{Z}$, and $Q: Px=d$ is $T=\{T_1,T_2,T_3,\dots,T_m\}\in\mathbb{Z}^{m\times m\times m}$ and $T_j$ is given by the matrix whose only nonzero entry is $(-1)^{x_j-2}$ in position $(j,j)$ (all other entries zero), for $j=1,2,\dots,m$,
where
\[
C:\ \big[X^TT_1X\ \ X^TT_2X\ \ X^TT_3X\ \cdots\ X^TT_mX\big]
\sbmat{x_1\\x_2\\x_3\\ \vdots\\ \vdots\\ x_m}
= [4k-1],
\]
and $k\in\mathbb{Z}$, and $V$: is given by $D_1=\operatorname{diag}\big((-1)^{x_1-2},(-1)^{x_2-2},(-1)^{x_3-2},\dots,(-1)^{x_m-2}\big)\in\mathbb{Z}^{m\times m}$, where $t\in\mathbb{Z}$ and
\[
Q: [X^TD_1X] = [4t-1].
\]
Now, if there exists a common solution among $S_L\cap S_Q\cap S_C = S=\{s_1,s_2,s_3,\dots,s_N\}\in\mathbb{Z}^m$ and if $s_i=(x_1,x_2,\dots,x_m)\in S$, then there exists a curve, nonsingular curve of genus $g$, where
\[
C_{s_i}: My^2 = x_1X^{x_1}+x_2X^{x_2}+\cdots+x_mX^{x_m}.
\]
If $y\ge1$, $x>1$, $\forall(x,y)\in C_{s_i}(\mathbb{Z}^+)$ holds then
\[
C_{s_i}(\mathbb{Z}^+) = \emptyset
\]
for every $s_i$.
\end{theorem}

\begin{proof}
According to THEOREM~\ref{thm5.3}, all elements of the system are equal: $A=(a_{ij})=P=(a_{ij})=B=(a_{ij})$ where $L: Ax=b$, $Q: Px=d$, and $V: Bx=m$; thus, let each element equal 1, where $a_{ij}=1$ for all $i,j$. This completes the proof.
\end{proof}

\section{Common Solutions of Algebraic Curves}

In this section, we generalize the relationship between solutions of linear equations and solutions of algebraic curves to a relationship among algebraic curves. Specifically, we consider multivariable polynomial curves and algebraic curves in two variables. We prove that if three curves in several variables admit a common solution in the integers, then one can construct a third algebraic curve that has no positive integer solutions. The
necessary conditions on the three curves, as well as the form of the resulting algebraic curve with no such solutions, are described in THEOREM 6.1 and 6.3.

\begin{theorem}\label{thm6.1}
Let $C_{P_i}$ be an algebraic curve, nonsingular curve of genus $g\ge 0$, where $\deg C_{P_i}=n$ and $\epsilon_j=(n-j)(n-j-2)$, where $b_0,b_1,\dots,b_n\in\mathbb{Z}$ are the coefficients of the curve and $U,K\ge 0\in\mathbb{Z}$, where $a,M\ge 1$ are odd numbers, all the prime divisors of which are of the form $p=4n+1$. Where $f_j\in\mathbb{Z}(x_1,x_2,x_3,\dots,x_n)$ and $T=\{C_1,C_2,C_3\}$, where $V(T)=\{P\in\mathbb{A}^n \mid C(P)=0 \text{ for all } C\in T\}$. Then, if
\[
C_1(P_i) = \sum_{j=0}^{n-1}(-1)^{n-j+1}b_jf_j(P_i) - 4k+1 = 0,
\]
\[
C_2(P_i) = \sum_{j=0}^{n}(-1)^{n-j+1}b_jf_j(P_i) - Mk^2 = 0,
\]
\[
C_3(P_i):\ \sum_{j=0}^{n-3}(-1)^{n-j-2}\epsilon_jb_jf_j(P_i) + b_{n-1}f_{n-1}(P_i) - 4U+1 = 0.
\]
If $P_i\in V(T)$, then there exists a curve
\[
C_{P_i}: MY^2 = b_0f_0(P_i)X^n + b_1f_1(P_i)X^{n-1} + b_3f_3(P_i)X^{n-2}+\cdots+b_nf_n(P_i).
\]
And if $y\ge1$, $x>1$, $\forall(x,y)\in C(\mathbb{Z}^+)$ holds then
\[
C_{P_i}(\mathbb{Z}^+) = \emptyset
\]
for every $P_i\in V(T)$.
\end{theorem}

\begin{proof}
According to THEOREM~\ref{thm3.2}, let $C: My^2 = a_0x^n+a_1x^{n-1}+\cdots+a_{n-1}x+a_n$ be an algebraic curve, nonsingular of genus $g$, with coefficients $a_0,a_1,a_2,\dots,a_n\in\mathbb{Z}$, where $M,k\ge 1$ are odd constants, such that all the prime divisors of $M,K$ satisfy $p\equiv 1\pmod 4$, and $C(\mathbb{Z}^+)$ are the positive integer solutions for curve $C$. Then if
\[
\sum_{j=0}^{n}(-1)^{n-j+1}a_j = Mk^2, \qquad \sum_{j=0}^{n-1}a_j(-1)^{n-j+1} = 4k-1,
\]
\[
\sum_{j=0}^{n-3}(-1)^{n-j-2}a_j(n-j)(n-j-2)+a_{n-1} = 4U-1, \qquad U,K\ge 0\in\mathbb{Z},
\]
if $y\ge1$, $x>1$, $\forall(x,y)\in C(\mathbb{Z}^+)$ holds then $C(\mathbb{Z}^+)=\emptyset$.

Now assume let $a_0=b_0f_0$ where $f_0\in\mathbb{Z}(x_1,x_2,x_3,\dots,x_n)$, and in the same manner let $a_1=b_1f_1$ where $f_1\in\mathbb{Z}(x_1,x_2,x_3,\dots,x_n)$. Now assume that for each $a_j=b_jf_j$, $j=1,2,\dots,n$ where $b_j\in\mathbb{Z}$ and $f_j\in\mathbb{Z}(x_1,x_2,x_3,\dots,x_n)$. As a result, we obtain the following equations
\[
\sum_{j=0}^{n}(-1)^{n-j+1}b_jf_j = Mk^2, \qquad \sum_{j=0}^{n-1}(-1)^{n-j+1}b_jf_j = 4k-1,
\]
\[
\sum_{j=0}^{n-3}(-1)^{n-j-2}(n-j)(n-j-2)b_jf_j + b_{n-1}f_{n-1} = 4U-1.
\]
Now assume that there exists a common solution to all of the curves $P\in\mathbb{A}^n_{\mathbb{Z}}$:
\[
C_1(P) = \sum_{j=0}^{n-1}(-1)^{n-j+1}b_jf_j(P) - 4k+1 = 0,
\]
\[
C_2(P) = \sum_{j=0}^{n}(-1)^{n-j+1}b_jf_j(P) - Mk^2 = 0,
\]
\[
C_3(P):\ \sum_{j=0}^{n-3}(-1)^{n-j-2}\epsilon_jb_jf_j(P) + b_{n-1}f_{n-1}(P) - 4U+1 = 0.
\]
Let $T=\{C_1,C_2,C_3\}$. Now, suppose the following exists: $V(T)=\{P\in\mathbb{A}^n \mid C(P)=0 \text{ for all } C\in T\}$ and $\epsilon_j=(n-j)(n-j-2)$. This means there is a solution for curves $C_1(P)$ and $C_2(P)$. Therefore, this satisfies the conditions of THEOREM~\ref{thm3.2} if $f_j(P)=a_j$. Thus, there exists a curve of degree $\deg C=n$, where
\[
C_P: MY^2 = b_0f_0(P)X^n + b_1f_1(P)X^{n-1}+b_3f_3(P)X^{n-2}+\cdots+b_nf_n(P).
\]
Then, if $y\ge1$, $x>1$, $\forall(x,y)\in C(\mathbb{Z}^+)$ holds then $C(\mathbb{Z}^+)=\emptyset$.

Now suppose there is another solution, $V(T)=\{P_1,P_2,\dots,P_N\}$. Then there are several curves that satisfy the conditions of THEOREM~\ref{thm3.2}:
\[
C_{P_i}: MY^2 = b_0f_0(P_i)X^n + b_1f_1(P_i)X^{n-1}+b_3f_3(P_i)X^{n-2}+\cdots+b_nf_n(P_i).
\]
And if $y\ge1$, $x>1$, $\forall(x,y)\in C(\mathbb{Z}^+)$ then
\[
C_{P_i}(\mathbb{Z}^+) = \emptyset
\]
for every $P_i\in V(T)$.
\end{proof}

\begin{theorem}\label{thm6.2}
Let $C_{P_i}$ be an algebraic curve, nonsingular of genus $g$, where $\deg C_{P_i}=n$ and $\epsilon_j=(n-j)(n-j-2)$, where $b_0,b_1,\dots,b_n\in\mathbb{Z}$ are the coefficients of the curve and $U,K\ge 0\in\mathbb{Z}$, where $a,M\ge 1$ are odd numbers, all the prime divisors of which are of the form $p=4n+1$. Where $f_j\in\mathbb{Z}(x_1,x_2,x_3,\dots,x_n)$ and $T=\{C_1,C_2,C_3\}$, where $V(T)=\{P\in\mathbb{A}^n \mid C(P)=0 \text{ for all } C\in T\}$. Then, if
\[
C_1(P_i) = \sum_{\substack{v_0+v_1+\cdots+v_k=n\\ 0\le j\le n-1}} (-1)^{n-j+1}b_ja_{v_0+\cdots+v_k}X^{v_0}X^{v_1}\cdots X^{v_k} - 4k+1 = 0,
\]
\[
C_2(P_i) = \sum_{\substack{v_0+v_1+\cdots+v_k=n\\ 0\le j\le n}} (-1)^{n-j+1}b_ja_{v_0+\cdots+v_k}X^{v_0}X^{v_1}\cdots X^{v_k} - Mk^2 = 0,
\]
\[
C_3(P):\ \sum_{\substack{v_0+v_1+\cdots+v_k=n\\ 0\le j\le n-3}} (-1)^{n-j+1}\epsilon_jb_ja_{v_0+\cdots+v_k}X^{v_0}X^{v_1}\cdots X^{v_k} + b_{n-1}a_{v_0+\cdots+v_k}X^{v_0}X^{v_1}\cdots X^{v_k} - 4U+1 = 0.
\]
If $P_i\in V(T)$, and $f_j(p_i) = \Big[b_ja_{v_0+\cdots+v_k}X^{v_0}X^{v_1}\cdots X^{v_k}\Big]^{n}_{\substack{j=0\\ v_0+v_1+\cdots+v_k=n}}$, then there exists a curve
\[
C_{P_i}: MY^2 = b_0f_0(P_i)X^n + b_1f_1(P_i)X^{n-1} + b_3f_3(P_i)X^{n-2}+\cdots+b_nf_n(P_i).
\]
And if $y\ge1$, $x>1$, $\forall(x,y)\in C(\mathbb{Z}^+)$ holds then
\[
C_{P_i}(\mathbb{Z}^+) = \emptyset
\]
for every $P_i\in V(T)$.
\end{theorem}

\begin{proof}
Let in THEOREM~\ref{thm6.1}, $f_j(p_i) = \Big[b_ja_{v_0+\cdots+v_k}X^{v_0}X^{v_1}\cdots X^{v_k}\Big]^{n}_{\substack{j=0\\ v_0+v_1+\cdots+v_k=n}}$. This completes the proof.
\end{proof}

\begin{theorem}\label{thm6.3}
Let $C_{P_i}$ be an algebraic curve where $\deg C_{P_i}=n$ and $\epsilon_j=(n-j)(n-j-2)$, where $b_0,b_1,\dots,b_n\in\mathbb{Z}$ are the coefficients of the curve and $U,K\ge 0\in\mathbb{Z}$, where $a,M$ are odd numbers, all the prime divisors of which are of the form $p=4n+1$. Where $f_j\in\mathbb{Z}(x_1,x_2,x_3,\dots,x_n)$ and $T=\{C_1,C_2\}$, where $V(T)=\{P\in\mathbb{A}^n \mid C(P)=0 \text{ for all } C\in T\}$. Then, if
\[
\sum_{j=0}^{n}(-1)^{f_j(P)+1}b_j = Mk^2, \qquad \sum_{j=0}^{n}(-1)^{f_j(P)+1}b_j = 4k-1,
\]
and
\[
C_1(P) = \sum_{j=1}^{n}(-1)^{f_j(P)-2}b_jf_j^{2}(P) - 4m+1 = 0,
\]
\[
C_2(P) = \sum_{j=1}^{n}(-1)^{f_j(P)-2}b_jf_j(P) + 2s = 0.
\]
If $f_j(P)=a$, then $a\in\mathbb{N}$; since $f_j(P)$ is always positive, a curve exists:
\[
C_{P_i}: MY^2 = a_1X^{f_1(P)}+a_1X^{f_2(P)}+a_2X^{f_3(P)}+\cdots+a_nX^{f_n(P)}.
\]
And if $y\ge1$, $x>1$, $\forall(x,y)\in C(\mathbb{Z}^+)$ then
\[
C_{P_i}(\mathbb{Z}^+) = \emptyset
\]
for every $P_i\in V(T)$.
\end{theorem}

\begin{proof}
According to THEOREM~\ref{thm3.2}, let $C: My^2=a_0x^n+a_1x^{n-1}+\cdots+a_{n-1}x+a_n$ be an algebraic curve, nonsingular of genus $g$, with coefficients $a_0,a_1,a_2,\dots,a_n\in\mathbb{Z}$, where $M,k$ are odd constants, such that all the prime divisors of $M,K$ satisfy $p\equiv 1\pmod 4$, and $C(\mathbb{Z}^+)$ are the positive integer solutions for curve $C$. Then, if
\[
\sum_{j=0}^{n}(-1)^{n-j+1}a_j = Mk^2, \qquad \sum_{j=0}^{n-1}a_j(-1)^{n-j+1} = 4k-1,
\]
\[
\sum_{j=0}^{n-3}(-1)^{n-j-2}a_j(n-j)(n-j-2)+a_{n-1} = 4U-1, \qquad U,K\ge 0\in\mathbb{Z},
\]
if $y\ge1$, $x>1$, $\forall(x,y)\in C(\mathbb{Z}^+)$ holds then $C(\mathbb{Z}^+)=\emptyset$. Note that in the following equation, $n$ represents the degree of the curve, $\deg C=n$:
\[
\sum_{j=0}^{n-3}(-1)^{n-j-2}a_j(n-j)(n-j-2)+a_{n-1} = 4U-1.
\]
Now assume that $a_j$ takes the following value: $a_0=a_1=a_2=a_3=\cdots=a_{k-1}=0$. As a result, we obtain
\[
a_0(n)(n-2) = a_1(n-1)(n-3) = a_2(n-2)(n-4) = \cdots = a_{k-1}(n-k+1)(n-k-1) = 0.
\]
Now assume $a_k(n-k) = b_1f_1$. So
\[
a_k(n-k)(n-k-2) = b_1f_1(f_1-2).
\]
In the same manner, let $a_{k+1}=a_{k+2}=a_{k+3}=a_{k+4}=\cdots=a_{k+r}=0$. As a result, we obtain
\[
a_{k+1}(n-k-1)(n-k-3) = a_{k+2}(n-k-2)(n-k-4) = a_{k+3}(n-k-3)(n-k-5) = \cdots = a_{k+r-1}(n-k-r+1)(n-k-r-2) = 0.
\]
Now assume $a_{k+r}(n-k-r) = b_2f_2$ where $f_2\in\mathbb{Z}(x_1,x_2,\dots,x_n)$. In the same manner, we iterate this process for certain values of $a_j$ until we obtain $f_1,f_2,f_3,\dots,f_n\in\mathbb{Z}(x_1,x_2,\dots,x_n)$. Now the equations are equal to
\[
\sum_{j=0}^{n}(-1)^{f_j-2}b_jf_j(f_j-2) = 4U-1,
\]
\[
\sum_{j=0}^{n}(-1)^{f_j+1}b_j = Mk^2, \qquad \sum_{j=0}^{n}(-1)^{f_j+1}b_j = 4k-1.
\]
Then
\[
\sum_{j=1}^{n}(-1)^{f_j-2}b_jf_j(f_j-2) = \sum_{j=1}^{n}(-1)^{f_j-2}b_jf_j^2 - 2\sum_{j=1}^{n}(-1)^{f_j-2}b_jf_j = 4U-1 = 4m+4s-1.
\]
Then let
\[
\sum_{j=1}^{n}(-1)^{f_j-2}b_jf_j^2 = 4m-1, \qquad \sum_{j=1}^{n}(-1)^{f_j-2}b_jf_j = -2s,
\]
\[
\sum_{j=0}^{n}(-1)^{f_j+1}b_j = Mk^2, \qquad \sum_{j=0}^{n}(-1)^{f_j+1}b_j = 4k-1.
\]
Now let $\mathbb{A}^n_{\mathbb{Z}} = \{(a_1,a_2,a_3,\dots,a_n)\mid a_i\in\mathbb{Z}\}$ and $T=\{C_1,C_2\}$ where
\[
C_1(P) = \sum_{j=1}^{n}(-1)^{f_j(P)-2}b_jf_j^{2}(P) - 4m+1 = 0.
\]
\[
C_2(P) = \sum_{j=1}^{n}(-1)^{f_j(P)-2}b_jf_j(P) + 2s = 0,
\]
\[
\sum_{j=0}^{n}(-1)^{f_j(P)+1}b_j = Mk^2, \qquad \sum_{j=0}^{n}(-1)^{f_j(P)+1}b_j = 4k-1.
\]
Now assume there is a solution to all those equations where $V(T)=\{P\in\mathbb{A}^n \mid C(P)=0 \text{ for all } C\in T\}$, on the condition that $f_j(P)=a\in\mathbb{N}$; this means $f_j(P)$ is always positive; thus, a curve exists:
\[
C_{P_i}: MY^2 = a_1X^{f_1(P)}+a_1X^{f_2(P)}+a_2X^{f_3(P)}+\cdots+a_nX^{f_n(P)}.
\]
And if $y\ge1$, $x>1$, $\forall(x,y)\in C(\mathbb{Z}^+)$ then
\[
C_{P_i}(\mathbb{Z}^+) = \emptyset
\]
for every $P_i\in V(T)$.
\end{proof}

\section{Intersections of Algebraic Curves}

In this section, we study the classification of algebraic curves that have no common intersection in the set of positive integers. We prove that, under certain prescribed properties of a set of integers, one can construct three algebraic curves from these integers such that the intersection of their solution sets over the positive integers is empty. In other words, the three curves have no common positive integer solution; this result is established in THEOREM 7.1 and 7.2. We then extend this result by incorporating tools from linear algebra to generate different classes of sets of integers that can be used in the classification of algebraic curves with no common intersection.

\begin{theorem}\label{thm7.1}
Let $C_1,C_2,C_3$ be three algebraic curves over $\mathbb{Z}$ with common coefficients $a_0,a_1,a_2,\dots,a_{2m}\in\mathbb{Z}$. Let $\deg C_1=2m$ and $\deg C_2=2m-1$ and $\deg C_3=2m-1$, and $C_1(\mathbb{Z}^+),C_2(\mathbb{Z}^+),C_3(\mathbb{Z}^+)$ are positive integer solutions to $C_1,C_2,C_3$ respectively. Then, if
\[
Mk^2 = -a_0+a_1-a_2-\cdots-a_{2m-1}-a_{2m}.
\]
And if $y\ge1$, $x>1$, $\forall(x,y)\in C_1(\mathbb{Z}^+),C_2(\mathbb{Z}^+),C_3(\mathbb{Z}^+)$, and $k\ne 0\in\mathbb{Z}$, where $M\ge 1$ is an odd positive integer, all prime divisors of $M$ of the form $p=4n+1$, then
\[
C_1(\mathbb{Z}^+)\cap C_2(\mathbb{Z}^+)\cap C_3(\mathbb{Z}^+) = \emptyset,
\]
where
\[
C_1: MY^4 + 2MY^2 = \sum_{j=0}^{2m}a_jX^j - M,
\]
\[
C_2: 4DY^p = \sum_{j=0}^{2m-1}a_jX^j + 1,
\]
\[
C_3: 4CY^q = \sum_{j=0}^{2m-3}\epsilon_ja_jX^j - a_{2m-1}X^{2m-1} - 1,
\]
and $q,p\ge 0\in\mathbb{N}$ where $C,D\in\mathbb{Z}$ and $\epsilon_j=(2m-j)(2m-j-2)$.
\end{theorem}

\begin{proof}
According to THEOREM~\ref{thm3.2}, let $C: My^2=a_0x^n+a_1x^{n-1}+\cdots+a_{n-1}x+a_n$ be an algebraic curve with coefficients $a_0,a_1,a_2,\dots,a_n\in\mathbb{Z}$, where $M,k$ are odd constants, such that all the prime divisors of $M,K$ satisfy $p\equiv 1\pmod 4$, and $C(\mathbb{Z}^+)$ are the positive integer solutions for curve $C$. Then, if
\[
\sum_{j=0}^{n}(-1)^{n-j+1}a_j = Mk^2, \qquad \sum_{j=0}^{n-1}a_j(-1)^{n-j+1} = 4K-1,
\]
\[
\sum_{j=0}^{n-3}(-1)^{n-j-2}(n-j)(n-j-2)a_j+a_{n-1} = 4U-1, \qquad U,K\ge0\in\mathbb{Z},
\]
if $y\ge1$, $x>1$, $\forall(x,y)\in C(\mathbb{Z}^+)$ holds then $C(\mathbb{Z}^+)=\emptyset$.

Now let the coefficients of the curve take the values $a_0=b_0x^0$ where $b_0,x^0\in\mathbb{Z}$, and $a_1=b_1x^1$ and $a_2=b_2x^2$ where $b_1,b_2,x^1,x^2\in\mathbb{Z}$, and similarly for all coefficients $a_j=b_jx^j\in\mathbb{Z}$. As a consequence, we obtain the following new equations:
\[
\sum_{j=0}^{n}(-1)^{n-j+1}b_jx^j = Mk^2, \qquad \sum_{j=0}^{n-1}(-1)^{n-j+1}b_jx^j = 4K-1,
\]
\[
\sum_{j=0}^{n-3}(-1)^{n-j-2}(n-j)(n-j-2)b_jx^j+b_{n-1}x^{n-1} = 4U-1.
\]
And then
\[
C: My^2 = b_0x^0x^n + b_1xx^{n-1} + b_2x^2x^{n-2}+\cdots+a_{n-1}x^{n-1}x+a_nx^n.
\]
Then
\[
C: My^2 = b_0x^n+b_1x^n+b_2x^n+\cdots+a_{n-1}x^n+a_nx^n.
\]
Now let $n=2m$ where $\deg C=2m$. Furthermore, let $(x,kx^m)\in C(\mathbb{Z}^+)$, where $x>1$, $k\ne0$. Substituting these into the equation of the curve $C$ yields
\[
C: M(kx^m)^2 = b_0x^{2m}+b_1x^{2m}+b_2x^{2m}+\cdots+a_{n-1}x^{2m}+a_nx^{2m}.
\]
Dividing both sides by $x^{2m}$ yields
\[
C: M(k)^2 = b_0+b_1+b_2+\cdots+b_{n-1}+b_n.
\]
Therefore, this implies that there exists a solution to the curve $C$ with $(x,kx^m)\in C(\mathbb{Z}^+)$ where $x>1$, $k\ne0$. Hence $C(\mathbb{Z}^+)\ne\emptyset$, which contradicts THEOREM~\ref{thm3.2}. This is a contradiction. Consequently, we conclude that the above equations are not valid:
\[
\sum_{j=0}^{2m}(-1)^{2m-j+1}b_jx^j \ne Mk^2
\]
or
\[
\sum_{j=0}^{2m-1}(-1)^{2m-j+1}b_jx^j \ne 4K-1
\]
or
\[
\sum_{j=0}^{2m-3}(-1)^{2m-j-2}(2m-j)(2m-j-2)b_jx^j+b_{2m-1}x^{2m-1} \ne 4U-1.
\]
This implies that at least one of the three equations must have no solution; otherwise there would exist $(x,kx^m)\in C(\mathbb{Z}^+)$ where $x>1$, $k\ne0$, which yields a contradiction. Note that $k,U\in\mathbb{Z}$. Now assume $K=Dy^p$ and $U=Cy^q$, and let $K=Y^2+1$ and $\epsilon_j=(2m-j)(2m-j-2)$. Then
\[
C_1: MY^4+2MY^2 = \sum_{j=0}^{2m}(-1)^{2m-j+1}a_jX^j - M,
\]
\[
C_2: 4DY^p = \sum_{j=0}^{2m-1}(-1)^{2m-j+1}a_jX^j + 1,
\]
\[
C_3: 4CY^q = \sum_{j=0}^{2m-3}(-1)^{2m-j-2}\epsilon_ja_jX^j - a_{2m-1}X^{2m-1} - 1.
\]
By contradiction, we deduce that at least one of the three equations has no solution. Now let $C_1(\mathbb{Z}^+),C_2(\mathbb{Z}^+),C_3(\mathbb{Z}^+)$ denote the sets of integral points on the three curves. Then
\[
C_1(\mathbb{Z}^+)\cap C_2(\mathbb{Z}^+)\cap C_3(\mathbb{Z}^+) = \emptyset.
\]
Now let $b_j=(-1)^{n-j+1}a_j$, so we have $(-1)^{2m-j+1}b_j = (-1)^{2m-j+1}(-1)^{2m-j+1}a_j = a_j$. Then
\[
C_1: MY^4+2MY^2 = \sum_{j=0}^{2m}a_jX^j - M,
\]
\[
C_2: 4DY^p = \sum_{j=0}^{2m-1}a_jX^j + 1,
\]
\[
C_3: 4CY^q = \sum_{j=0}^{2m-3}\epsilon_ja_jX^j - a_{n-1}X^{2m-1} - 1.
\]
Then
\[
M(y)^2 = -a_0+a_1-a_2-\cdots-a_{2m}-a_{2m}. \qedhere
\]
\end{proof}

\begin{theorem}\label{thm7.2}
Let $C_1,C_2,C_3$ be three algebraic curves over $\mathbb{Z}$ with common coefficients $b_1,b_2,b_3,\dots,b_{2m}\in\mathbb{Z}$. Let $\deg C_1=2m$ and $\deg C_2=2m-1$ and $\deg C_3=2m-1$, and $C_1(\mathbb{Z}^+),C_2(\mathbb{Z}^+),C_3(\mathbb{Z}^+)$ are positive integer solutions to $C_1,C_2,C_3$ respectively. Then, if
\[
S_M: My^2 = -b_0x_0+b_1x_1-b_2x_2-\cdots-b_{2m-1}x_{2m-1}-b_{2m}x_{2m}.
\]
If $S_M=\{P_1,P_2,\dots,P_N\}\subset\mathbb{Z}^{2m}$ are solutions of the linear equation $S_M$, and if $y\ge1$, $x>1$, $\forall(x,y)\in G_1(\mathbb{Z}^+),C_2(\mathbb{Z}^+),F_3(\mathbb{Z}^+)$ and $k\ne0\in\mathbb{Z}$, where $M\ge1$ is an odd positive integer, all prime divisors of $M$ of the form $p=4n+1$, then
\[
G_{P_1}(\mathbb{Z}^+)\cap C_{P_1}(\mathbb{Z}^+)\cap F_{P_1}(\mathbb{Z}^+) = \emptyset.
\]
If $P_1=(x_1,x_2,\dots,x_{2m})\in S_M$, then
\[
G_{P_1}: MY^4+2MY^2 = \sum_{j=0}^{2m}b_jx_jX^j - M,
\]
\[
C_{P_1}: 4DY^p = \sum_{j=0}^{2m-1}b_jx_jX^j + 1,
\]
\[
F_{P_1}: 4CY^q = \sum_{j=0}^{2m-3}\epsilon_jb_jx_jX^j - b_{2m-1}x_{2m-1}X^{2m-1} - 1,
\]
and $q,p\ge0\in\mathbb{N}$ where $C,D\ge1\in\mathbb{N}$ and $\epsilon_j=(2m-j)(2m-j-2)$. If there exist infinitely many solutions in the case, denoted by $N\Rightarrow\infty$, then
\[
\Big\{G_{P_j}(\mathbb{Z}^+)\cap C_{P_j}(\mathbb{Z}^+)\cap F_{P_j}(\mathbb{Z}^+)\Big\}^{\infty}_{j=1} = \emptyset \quad \text{if } N\Rightarrow\infty.
\]
\end{theorem}

\begin{proof}
Let, in THEOREM~\ref{thm7.1}, $a_j=b_jx_j$ where $b_j,x_j\in\mathbb{Z}$. Then
\[
My^2 = -b_0x_0+b_1x_1-b_2x_2-\cdots-b_{2m-1}x_{2m-1}-b_{2m}x_{2m}.
\]
Thus we obtain new equations equal to the following, according to THEOREM~\ref{thm7.1}. Note that
\[
G_{P_1}: MY^4+2MY^2 = \sum_{j=0}^{2m}b_jx_jX^j - M,
\]
\[
C_{P_1}: 4DY^p = \sum_{j=0}^{2m-1}b_jx_jX^j + 1,
\]
\[
F_{P_1}: 4CY^q = \sum_{j=0}^{2m-3}\epsilon_jb_jx_jX^j - b_{2m-1}x_{2m-1}X^{2m-1} - 1.
\]
And if $y\ge1$, $x>1$ and $\forall(x,y)\in G_{P_1}(\mathbb{Z}^+),C_{P_1}(\mathbb{Z}^+),F_{P_3}(\mathbb{Z}^+)$ and $k\ne0\in\mathbb{Z}$, where $M$ is an odd positive integer, all prime divisors of $M$ of the form $p=4n+1$, then
\[
G_{P_1}(\mathbb{Z}^+)\cap C_{P_1}(\mathbb{Z}^+)\cap F_{P_1}(\mathbb{Z}^+) = \emptyset.
\]
Now suppose that $S_M=\{P_1,P_2,\dots,P_N\}\subset\mathbb{Z}^{2m}$ is a solution of the equation
\[
My^2 = -b_0x_0+b_1x_1-b_2x_2-\cdots-b_{2m-1}x_{2m-1}-b_{2m}x_{2m}.
\]
Now if $P_2=(x_0,x_1,\dots,x_m)\in S_M$ then, by THEOREM~\ref{thm7.1}, we have
\[
G_{P_2}(\mathbb{Z}^+)\cap C_{P_2}(\mathbb{Z}^+)\cap F_{P_2}(\mathbb{Z}^+) = \emptyset.
\]
Suppose there exist infinitely many solutions, where $N\Rightarrow\infty$, so
\[
\Big\{G_{P_j}(\mathbb{Z}^+)\cap C_{P_j}(\mathbb{Z}^+)\cap F_{P_j}(\mathbb{Z}^+)\Big\}^{\infty}_{j=1} = \emptyset \quad \text{if } N\Rightarrow\infty. \qedhere
\]
\end{proof}

\end{document}